\documentclass[12pt]{article}

\usepackage{latexsym}
\usepackage{epsfig}
\usepackage{enumerate,amsmath,amsthm,latexsym,amssymb,hyperref}
\usepackage{bbm}
\usepackage[active]{srcltx}
\usepackage{enumerate}
\usepackage{float}

\usepackage[usenames]{xcolor}

\usepackage{caption}
\usepackage{subcaption}

\usepackage[nameinlink,capitalise]{cleveref}
\crefname{theorem}{Theorem}{Theorems}
\crefname{lemma}{Lemma}{Lemmas}
\crefname{equation}{}{}

\usepackage{graphicx}

\def\cT{{\mathcal T}}
\def\blue{\color{blue}}

\def\cyan{\color{cyan}}

\def\red{\color{red}}

\def\purple{\color{purple}}

\newcommand{\brac}[1]{\left( #1 \right)}
\newcommand\bfrac[2]{\left(\frac{#1}{#2}\right)}
\def\E{{\bf E}}

\def\Pr{\mathbb{P}}
\def\rai{\rightarrow \infty}

\def\ooi{(1+o(1))}

\newcommand{\ignore}[1]{ }

\def\a{\alpha}

\def\d{\delta}
\def\D{\Delta}
\def\e{\varepsilon}

\def\th{\theta}
\def\Th{\Theta}
\def\l{\lambda}

\def\om{\omega}
\def\Om{\Omega}

\newcommand{\ra}{\rightarrow}

\newcommand{\imp}{\implies}

\def\sm{\! \setminus \!}
\newcounter{rot}

\def\half{{\scriptstyle \frac 12}}
\newcommand{\floor}[1]{\lfloor #1 \rfloor}
\newcommand{\ceil}[1]{\lceil #1 \rceil}

\newcommand{\TheoremGnp}{}
\newcommand{\TheoremLine}{}
\newcommand{\TheoremKnSync}{}
\newcommand{\LemmaGNPStructure}{}

\newtheorem{theorem}{Theorem}
\newtheorem{lemma}[theorem]{Lemma}
\newtheorem{corollary}[theorem]{Corollary}

\begin{document}

%
\title{Discrete Incremental  Voting}

\author{Colin Cooper\thanks{Department of  Informatics,
King's College London, London WC2R 2LS, UK.}
\and Tomasz Radzik$^*$ 
\and Takeharu Shiraga\thanks{
Department of Information and System Engineering, Chuo University, Tokyo, Japan.
Research supported by JSPS KAKENHI Grant Number 19K20214.}
}
\maketitle \makeatother

\begin{abstract}
We consider a type of pull voting suitable for discrete numeric opinions which can be compared on a linear scale, for example, 
1 ('disagree strongly'), 2 ('disagree'), $\ldots,$ 5 ('agree strongly'). On observing the opinion of a random neighbour, a vertex changes its opinion incrementally towards the value of the neighbour's opinion, if different. For opinions drawn from a set $\{1,2,\ldots,k\}$, 
the opinion of the vertex would change by $+1$ if the opinion of the neighbour is larger, or by $-1$, if it is smaller.

It is not clear how to predict the outcome of this process,
but we observe that the total weight of the system,
that is, the sum of the individual opinions of all vertices, is a martingale. This allows us analyse the outcome of the process on some classes of dense expanders such as clique graphs $K_n$ and random graphs $ G_{n,p}$
for suitably large $p$. 
If the average of the original opinions  
satisfies $i \le c \le i+1$ for some integer $i$, then the asymptotic probability that opinion $i$ wins is  $i+1-c$, and the probability that opinion $i+1$ wins is $c-i$.
With high probability, the winning opinion cannot be other than $i$ or $i+1$.

To contrast this, we show that for a path
and opinions $0,1,2$ arranged initially in
non-decreasing order along the path, 
the outcome is very different.
Any of the opinions can win with constant probability, provided that each of the two extreme opinions $0$ and $2$ is initially supported by a constant fraction of vertices.  
\end{abstract}

\maketitle

\section{Introduction}

\paragraph{Background on distributed pull voting.}
Distributed voting has  applications in various
fields of computer science including consensus and leader election in large networks
\cite{BMPS04,HassinPeleg-InfComp2001}.
Initially, each vertex has some value chosen from a set $S$, and the aim is 
that the vertices reach consensus on 
(converge to) the same value, which 
should, in some sense, reflect the 
initial distribution of the values.
Voting algorithms are usually simple,
fault-tolerant, and easy to implement \cite{HassinPeleg-InfComp2001,Joh89}.

Pull voting is a  simple form of distributed voting.
At each step, a randomly chosen vertex (asynchronous process), or
each vertex (synchronous process), replaces its opinion with that of  randomly chosen neighbour.
The probability a particular opinion, say opinion $A$, wins
is $d(A)/2m$, where $d(A)$ is the sum of the degrees of the vertices initially holding opinion $A$, and 
$m$ is the number of edges in the graph; see Hassin and Peleg~\cite{HassinPeleg-InfComp2001}
and Nakata {\em et al.}~\cite{Nakata_etal_1999}.
The pull voting process can be modified to consider two or more opinions at each step. The aim of this modification is twofold; to ensure the majority (or plurality) wins, and to speed up the run time of the process.
Work on {\em best-of-$k$} models, where 
a vertex replaces its opinion with the opinion most represented in a sample of $k$ opinions, includes 
\cite{AD,Becchetti, Bec2, Bec3, Petra,CER,CRRS,Ghaffari,NanNico,NS}.

The general model of pull voting regards the opinions as incommensurate, and thus not comparable on a numeric scale. In contrast to this, Doerr at al. \cite{Doerr} consider opinions drawn from an ordered set and a process which aims to converge to the median. At each step a random vertex selects two neighbours and replaces its opinion by the median of all three values (including its own current
value).

In this paper we consider another such variant of pull voting,  in which the opinions are comparable on a linear scale.
This variant, which we call {\em discrete incremental voting},
can be seen as modelling the convergence to consensus of a group opinion, based on compromise during extended discussion. 
Thus the final opinion may not be one held originally.
For any connected graph, if the initial 
{(degree-weighted)} average of the opinions is $c$, 
{then the expectation  of the final 
opinion (a random variable)
is always equal to the initial average $c$. Furthermore, for some classes of expanders, the process  converges to an integer average (either $\floor{c}$ or $\ceil{c}$) with high probability.}
Seen in this context, the pull voting processes above mirror the statistical measures of {Mode, Median and Mean, for pull voting, median voting
and discrete incremental voting, 
respectively}.

\paragraph{Discrete incremental voting: An introduction.}
We assume the initial opinions of the vertices are chosen from among the integers  
$\{1,2,...,k\}$.
As a simple example, suppose the entries   reflect the views of the vertices about some issue, and range from  1 ('disagree strongly') to $k$ ('agree strongly'). Then it seems unrealistic that a vertex would change its opinion to that of a neighbour, (as in pull voting), based only on observing what the neighbour thinks.
However, people being what they are, it seems possible that they may modify their opinion slightly towards the opinion of their neighbour on observing it.

In the simplest case,  suppose that a vertex $v$ has opinion $i$ and observes at its
neighbour $u$ opinion $j$. If $j>i$, then
vertex $v$ modifies its opinion to $i+1$ (tends to agree more). Similarly, if the observed neighbour $u$ has value $j<i$,
vertex $v$ changes its opinion to $i-1$ (tends to disagree more).
The neighbour $u$ does not change its opinion
at this interaction.
That this process converges, and the value it converges to, is the topic of this paper.

{We  consider two related asynchronous and one synchronous variants of incremental voting. 
Given a connected graph
$G = (V,E)$ with $n$ vertices
and $m$ edges,
let $X(t)=(X_v(t): v\in V)$ be the vector of integer opinions at step $t$.
The value of $X(t+1)$ is obtained as follows.
}

{\sc Asynchronous vertex process}: \;
Given $X=X(t)$, pick a vertex $v$ uniformly at random (u.a.r.) and an adjacent 
edge $(v,w)$ u.a.r.
The following update rule  $X_v\ra X'_v $ holds,

\begin{equation}\label{Tab}
\left.
\begin{array}{lll}
X_v <X_w & \imp & X'_v=X_v+1\\
X_v=X_w & \imp & X'_v=X_v\\
X_v > X_w & \imp & X'_v=X_v-1
\end{array}
\; \right\}
\end{equation}


{\sc Asynchronous edge process%
\footnote{The edge process can be seen as a vertex process where the leading vertex $v$ is sampled with probability $\pi_v=d(v)/2m$ rather than uniformly at random.}}: \; 
Pick a random endpoint $v$ of a random edge $e=(v,w)$ selected u.a.r. The value $X_v$ at vertex $v$ is updated to $X_v'$
using  the rules  in \eqref{Tab} above.

{

{\sc Synchronous vertex process}: \;
Given $X=X(t)$,  each vertex $v$ picks an adjacent 
edge $(v,w)$
u.a.r.\
and updates its value $X_v\ra X'_v $ according to~\eqref{Tab}.
}

\paragraph{Discrete incremental voting: Main results.}
If the initial set of opinions is $\{0,1\}$ (or $\{i,i+1\}$ for an integer $i$)
the incremental voting (with updates~\eqref{Tab}) is equivalent to ordinary 'two-value' pull voting as studied by
\cite{HassinPeleg-InfComp2001} and others: when a vertex $v$ updates its value using the value at a neighbour $w$, vertex $v$ simply takes on the value from vertex $w$.
The 'two-value' pull voting comes in the same three variants: asynchronous or synchronous vertex process, or
asynchronous edge process.
In a vertex process (asynchronous or synchronous), the probability that opinion $0$ wins, $d(0)/2m$, is proportional to the sum $d(0)$ of the degrees of the vertices initially holding this opinion.
The simplest  case which differs from pull voting has opinion values in $\{0,1,2\}$.
In general we assume the initial values are in the range $\{0, 1, ..., k\}$ or $\{1, ..., k\}$, 
where appropriate bounds on $k$ may be required in the analysis. 

In order to reach a consensus opinion, all other opinions must be eliminated. The only way to irreversibly reduce the number of opinions, is to remove one of the extreme values in the order, leading to the next stage. 
{The process continues through such stages}
until one opinion remains. 
Returning to our original example, 1 ('disagree strongly'), 2 ('disagree'), 3 ('indifferent'), 4 ('agree'), 5 ('agree strongly'),
suppose we start with initial opinions $\{1,2,5\}$.  Then a possible evolution of the system is
\[
\{1,2,5\}\ra \{1,2,3,4\} \ra \{2,4\} \ra \{2,3\}\ra \{3\},
\]
where 
{the sets of opinions at the beginning of each stage are indicated, and}
each '$\ra$' represents a sequence of one or more steps constituting one stage.
The intermediate values may disappear but then they appear again (in the above example,
opinion $3$ disappeared in stage 2 and they appeared again in stage 3).
Eventually, as extreme values disappear, we reach the final stage of voting when  only two adjacent values  remain. In the example above the final stage has values $\{2,3\}$.
At this point the process reverts to ordinary two-value pull voting.  Suppose only values $\{i,i+1\}$ remain.
{Let $A_j$, $j \in\{i,i+1\}$, be the set of vertices with value $j$ at the start of this final stage, and 
$N_j=|A_j|$, so $N_i + N_{i+1}=n$. Let $d(A)=\sum_{v \in A} d(v)$ be the total degree of set $A$. } The probability  
that $i$ wins is 
\begin{equation} \label{probw}
\Pr(i \text{ wins})=\frac{N_i}{n} \quad (\text{Edge process}), \quad \quad
{\Pr(i \text{ wins})=\frac{d(A_i)}{2m} \quad(\text{Vertex process}).}
\end{equation}

In  Section \ref{Sec2}, Lemma \ref{Lemma1}, we prove that the average weight of the process is a martingale in both the asynchronous and synchronous  processes. 
(The vertex opinion, value, and weight refer to the same 
quantity.)
This allows us to establish Theorem \ref{ThA} which gives the distribution of winning opinions in the case where the initial average $c$ is maintained throughout the process. In this case, when only two opinions $\{i,i+1\}$ remain, we have $i \le c \le i+1$ and their winning probability is determined as in \eqref{probw}  above. 

It is shown in Section \ref{Sec2}, Lemma~\ref{T2}, that the expected time 
for one of the two extreme opinions to disappear is $O(T_2)$, where $T_2$ is the (worst-case) expected time to consensus for two-value pull voting on the same graph. Thus the expected time to consensus is $O(k T_2)$.
See \cite{CERO} and \cite{Peres} for graph specific bounds on  the value of $T_2$. 
 However, for the clique graph $K_n$ and some 
 classes of expanders%
 \footnote{We view expansion in terms of the relative number of 
 edges between sets $S$ and $V\setminus S$.}, 
 as the number of opinions $k$ increases, the bound $O(k T_2)$ on consensus time becomes weak. 
 For such graphs, we can show that with high probability the extreme values disappear faster than the time to complete two-value pull voting.
 Thus, with suitable bounds on $k$, the expected run time can be reduced from 
$O(k T_2)$ to $O(T_2)$, and is directly comparable with ordinary pull voting.
See e.g.,  Lemma \ref{lem:extreme}.
Ideally (for easier analysis) we would like one of the two extreme opinions to disappears completely before 
moving on to considering the next extreme opinion. 
However, to obtain good bounds, in some cases we have to move on to the next extreme opinion, say the next smallest opinion $\kappa$, while 
some small number of vertices may still hold opinions smaller than $\kappa$. 

As the average opinion is a martingale {(details of this are in the next Section~\ref{Sec2})}, 
in cases where the process converges rapidly to two neighbouring states, martingale concentration allows us to  use Theorem \ref{ThA} to predict the outcome of the process. This
is fundamental for our analysis  on expander graphs. For the cases we studied, $G_{n,p}$ and $K_n$, essentially the process converges quickly to two neighbouring states $\{i,i+1\}$. 
Because the time to consensus in the final stage is  determined by known results, i.e., two-value pull voting, we only need to estimate the time to reach a final pair of values $\{i,i+1\}$ where w.h.p. $i \le c \le i+1$. The overall expected time to consensus is determined by the (slower) final stage of two-value pull voting; namely $O(n)$ for the synchronous, and $O(n^2)$ for the asynchronous process. 

{We illustrate incremental voting
using three examples:} 
the asynchronous process on 
$G_{n,p}$ (Theorem~\ref{ThGnp-intro}) and the synchronous process on $K_n$ (Theorem~\ref{ThKnSync}), both of which work as one might expect, and an asynchronous process on the  path which does not 
(Theorem~\ref{TheLine-intro}). 

\paragraph{Notation.}
For functions $a=a(n)$ and $b=b(n)$, $a \sim b$  denotes $a=b(1+o(1))$, 
where $o(1)$ is a function of $n$ which tends to zero as $n \rai$. We use $\om$ to denote a generic quantity tending to infinity as $n \rai$, but suitably slowly as required in the given proof context.  
An event $A$ on an $n$-vertex graph holds with high probability (w.h.p.),
if $\Pr(A)=1-o(1)$. 



\renewcommand{\TheoremKnSync}{%
{\sc Synchronous incremental voting on $K_n$}. \\
Let the initial values {of the vertices of $K_n$} be chosen from $\{1,2,\dots,k\}$, where $k=o(n/(\log n)^2)$, and 
let $S(0)= \sum_{v\in V}X_v(0) =cn$. \;\;
\begin{enumerate}[(i)]
\item
If $i < c < i+1$, then $\,\Pr(i \text{ wins}) \sim i+1-c$ and $\,\Pr(i+1 \text{ wins}) \sim c-i$. 
If $c =i (1+o(1))$, then $\Pr(i \text{ wins}) \sim 1$.
\item
The number of opinions is reduced to at most three consecutive values 
in $O(k \log n)$ steps w.h.p., and the expected time for the whole process to finish is $O(n)$.
\end{enumerate}
}

\begin{theorem}\label{ThKnSync}
\TheoremKnSync
\end{theorem} 

 A similar analysis  for the asynchronous process, giving w.h.p. convergence to three  adjacent values in $O(nk \log n)$ steps is given in Appendix~\ref{KnAsync}.  The expected time for the asynchronous process to finish is $O(n^2)$.




\renewcommand{\TheoremGnp}{%
{\sc Asynchronous incremental voting on random graphs.}\\
Let $G \in G_{n,p}$, where $np \ge \log^{1+\e} n$ for some constant $\e>0$.
Let the initial values be in $\{1,2,...,k\}$ where
$k$ is a fixed positive integer, and $S(0)=\sum_{v\in V}X_v(0) = cn$
be the initial total weight.
\begin{enumerate}[(i)]
\item
If $i < c <i+1$,
then 
$
\Pr(i \text{ wins}) \sim i+1-c$, and $\Pr( i+1 \text{ wins}) \sim c-i$. 
If $c =i (1+o(1))$, then $\Pr(i \text{ wins}) \sim 1$.
\item
The expected time for the asynchronous process to finish is $O(n^2)$.
\end{enumerate}
}

\begin{theorem}
\label{ThGnp-intro}
\TheoremGnp
\end{theorem}

{The intuitive basis of Theorems~\ref{ThKnSync} and~\ref{ThGnp-intro}, is to prove that the \lq extremal\rq~ values from $\{1,2,...,k\}$ disappear rapidly leaving just two values $i,i+1$ whose weighted average is  $c$, and to which we can apply the results of two-value pull voting. For $K_n$ this is essentially what happens, and for $G_{n,p}$ it is a reasonable approximation.}
We remark that the expected time to complete two-value pull voting on $K_n$ and $G_{n,p}$ is $\Th(n)$ (synchronous), and $\Th(n^2)$ (asynchronous), see \cite{AFill} Chapter 14.3.3 or  \cite{CERO}; {and  the completion time of incremental voting is asymptotically of the same order.}


To complement the above results, for graphs which are not expanders, we give an example on the path graph for which the final answer is quite different than in Theorems~\ref{ThKnSync} and~\ref{ThGnp-intro}.
%
Let $P_n$ be the path with vertex set $\{1,2,...,n\}$, with initial values $\{0,1,2\}$  ordered on the path vertices in non-decreasing value: first $N_0\ge 0$ zeroes, then $N_1\ge 0$ ones and finally $N_2\ge 0$ twos,
where $N_0 + N_1 + N_2 = n$.
We refer to such an arrangement as the ordered path.

\renewcommand{\TheoremLine}{%
{\sc Asynchronous incremental voting on the ordered path $P_n$. }
\\
If initially $N_0=an, N_1= (1-(a+b))n$, and $N_2=bn$, then
\begin{eqnarray*}
\Pr(\text{Opinion 0 wins})& \sim &a(1-b)\\
\Pr(\text{Opinion 1 wins})& \sim & ab+(1-a)(1-b)\\
\Pr(\text{Opinion 2 wins})& \sim &(1-a)b.
\end{eqnarray*}
}

\begin{theorem}
\label{TheLine-intro} 
\TheoremLine
\end{theorem}

{\bf An example for comparison. }\;\; 
We consider  values in $\{0,1,2\}$, where  initially $1/5$ of the values are zero, none are one, and $4/5$ are two. Thus $c=8/5$ and $1 < c <2$. \\
In $K_n$ and $G_{n,p}$, 
\[ \Pr(0 \text{ wins}) \sim 0, \qquad \Pr(1 \text{ wins}) \sim 2/5, \qquad \Pr(2 \text{ wins}) \sim 3/5,
\]
whereas on the ordered path
\[ \Pr(0 \text{ wins}) \sim 1/25, \qquad \Pr(1 \text{ wins}) \sim 8/25, \qquad \Pr(2 \text{ wins}) \sim 16/25.
\]



\section{Basic properties of  incremental voting }\label{Sec2}

{Let $X(t)=(X_v(t): v\in V)$ be the vector of integer opinions held by the vertices at step~$t$,
where $X(0)$ is the vector of initial opinions.
We use the notation $A_i(t) = \{v\in V: X_v(t) =i\}$ for the set of vertices holding opinion $i\in \{1,...,k\}$ at time~$t$,
and let $N_i(t)=|A_i(t)|$. 
For notational convenience, we may abbreviate by dropping the step index $t$, for example, 
$N_i$ and $N'_i$ would refer to the number of vertices holding opinion $i$ at the beginning and at the end, 
respectively, of the current step.
Let $S(t)$ be the total weight  at step $t \ge 0$: 
$S(t)=\sum_{v\in V} X_v(t)=\sum_j j N_j(t)$. 
The average of the initial opinions $c=S(0)/n$.}  
Let $\pi_v=d(v)/2m$ where $m$ is the number of edges of the graph,
and let $Z(t)=n \sum_{v\in V} \pi_v X_v(t)$ be the degree biased weight. 
{We note that for regular graphs, $\pi_v=1/n$, in which case $S(t)=Z(t)$.}
We also use notation
$\|\pi\|_2 = \sqrt{\sum_i \pi^2}$ and $\|\pi\|_\infty = \max_{v\in V}\pi_{v}$.

{A random variable $W(t), t=0,1,...$  of the incremental voting process is a martingale if its expected value at the next step depends only on the current opinions $X(t)$, and it satisfies $\E (W(t+1) \mid X(t)) = W(t)$. 
}

\begin{lemma}\label{Lemma1}\label{SKn}\label{Lemma2}{\sc The average weight is a martingale.}\;\;
The following hold for each $t\ge 0$.

\begin{enumerate}[(i)]
\item
{\bf Asynchronous edge process.}   For arbitrary graphs, $S(t)$ is a martingale.
\item
{\bf Asynchronous vertex process. }  For arbitrary graphs, $Z(t)$ is a martingale.
\item {\bf Synchronous vertex process. }  For {arbitrary} graphs,  
{$Z(t)$} is a martingale.
\end{enumerate}
\end{lemma}
\begin{proof}
{\em Proof of (i).} 
Consider step $t+1$, take any edge $(v,w)$, and let $\D_v(w)$ be the change in $X_v$ 
if this edge and its endpoint $v$ are chosen in step $t+1$. 
Thus $\D_v(w) \in \{-1,0,+1\}$ and, (see  \eqref{Tab}), $\D_v(w)=-\D_w(v)$. 
Only one of these changes can occur at a given step in the asynchronous process.
Let $ e=(v,w)$ be the chosen edge, an event of probability $1/m$ in the edge process. 
Then,
\[
\E(S(t+1) \mid X(t), e=(v,w) \text{ chosen})= S(t)+ \frac12 \D_v(w) +\frac12 \D_w(v) =S(t).
\]
{\em Proof of (ii).}
Let $A_i(t)$ be the vertices with value $i$ at step $t$. 
For a vertex $u \in A_i$, let $s_{ij}(u)$ be the number of edges from $u$ to $A_j$. Each edge has two ends, so {adding the edges between $A_i$ and $A_j$ in two ways,}
\begin{equation}\label{NiNj}
\sum_{u \in A_i} s_{ij}(u)=\sum_{v \in A_j} s_{ji}(v).
\end{equation}
For $1 \le i < j \le k$, let $\D_{ij}$ be the change in $Z(t)$ at step $t+1$ arising from an interaction 
between $A_i$ and $A_j$ (a vertex with opinion $i$ picks up a neighbour with opinion $j$, or vice versa).
We have $Z(t+1) = Z(t) + \sum_{i<j} \D_{ij}$.
{In the vertex process, we first pick a vertex $u$ u.a.r.\ and then an edge $(u,v)$ from $u$ u.a.r.. If $u \in A_i$ is the sampled vertex, the probability an edge from $u$ to $A_j$ is chosen is $s_{ij}(u)/d(u)$. 
As $\pi_u=d(u)/2m$, and then using~\eqref{NiNj},}
\begin{equation}\label{expected-change}
\E \D_{ij}=\; \sum_{u \in A_i} \frac 1n \frac{s_{ij}(u)}{d(u)} \pi_u
- \sum_{v \in A_j} \frac 1n \frac{s_{ji}(v)}{d(v)} \pi_v  \; = \;
 \frac 1{2nm} \brac{\sum_{u \in A_i} s_{ij}(u)-\sum_{v \in A_j} s_{ji}(v)}\;=\;0.
\end{equation}
{\em Proof of (iii).}
{In the synchronous vertex process,
with the notation as above in {\it (ii)},
the expected value of $\D_{ij}$ is as
in~\eqref{expected-change} but 
without the $1/n$ factors (since each vertex selects a neighbour and updates 
its value).} 
\ignore{
 Let $N_i$ be the number of vertices with opinion $i$. For brevity, we give a proof for $K_n$.
\begin{align*}
    \E\left[S(t+1)-S(t)\mid X(t)\right]
    &=\sum_{v\in V}\E\left[X_v(t+1)-X_v(t)\mid X(t)\right]\\
    &=\sum_{v\in V}\left(\sum_{i=X_v(t)+1}^{k}\frac{N_i}{n}-\sum_{i=1}^{X_v(t)-1}\frac{N_i}{n}\right)\\
    &=\frac 1n \sum_{j=1}^k N_j \brac{\sum_{i=j+1}^k N_i-\sum_{i=1}^{j-1} N_i}=0.
\end{align*}
}
\end{proof}

We note that Theorem~\ref{SKn} applies to arbitrary graphs, not necessarily connected. The martingales $S(t)$ and $Z(t)$ converge for arbitrary graphs, as noted in the next theorem, but if the graph is not connected, then each component converges to its own value.
We assume that the graph is connected,
when we analyse convergence to one value.

As the process is randomized, the final value (if the graph is connected) is a  {random variable with} distribution $D(i)$ on the initial values $\{1,...,k\}$, 
where $D(j)=\Pr(j \text{ wins})$.
The following theorem helps us to characterize this distribution in certain cases. 
If only two neighbouring opinions $i, i+1$ remain at some step $t$, the process is equivalent to two-value pull voting, and we say the incremental voting is at the final stage. 





\begin{theorem}\label{ThA}{\sc Distribution of winning value.}\;\;
Let $W(t)=S(t)$ when referring to
the edge model, and let $W(t)=Z(t)$, when referring to the vertex model. 
Let {$W(0)=cn$} be the total initial weight, 
where $n$ is the number of vertices in the graph and $c$ is the initial average opinion.
\begin{enumerate}[(i)]
\item { 
For an arbitrary graph, 
the expected average opinion at any step is always the initial average: $\E[W(t)/n] = W(0)/n = c$. 
The process $W(t)$ converges to a time invariant random variable.}
\item
For a connected graph,
if at the start of the final stage only two opinions $i$ and $i+1$ remain and the total weight $W$ is {$c'n$}, 
then for any connected graph,
the winning opinion is $i$ with probability $p=i+1-c'$, or $i+1$
with probability $q=c'-i$.
\item For a connected graph, 
suppose the final stage is reached in $T$ steps, 
where $T = o(1/\| \pi\|^2_2)$ for the synchronous vertex process (which reduces to $T=o(n)$ for regular graphs),
$T= o(n^2)$ for the asynchronous edge process,
and $T = o(1/\| \pi\|^2_\infty)$ for the asynchronous vertex process.
Then w.h.p. $W(T) \sim cn$ and the results of part (ii) hold with $c' \sim c$.
That is, for $i$ such that $i\le c < i+1$, 
the winning opinion is $i$ with probability $p\sim i+1-c$, and is $i+1$
with probability $q\sim c-i$.
\end{enumerate}
\end{theorem}

\begin{proof}
{\bf (i)}
{The first part follows from $\E W(t)=W(0)$ (Lemma \ref{Lemma1}).  
$\E W^2(t) \le k^2$ and the limit random variable, for $t\rai$, exists by the martingale convergence theorem.}

\noindent
{\bf (ii)}
Using~\eqref{probw}, we have $ipn + (i+1)qn = W$, 
implying that $p=i+1-c'$ and $q=c'-i$.

\ignore{
Let $x_1, x_2, ... , x_n$ be the opinions of the nodes.
Each $x_i$ is in $\{1,2,...,k\}$, where $k$ can be large,
but constant (as $n$ increases to infinity).

Let's look only at the edge process.
In one step, a random edge $(i,j)$ is chosen and then
either $ x_i + x_j$ doesn't change (if $x_i = x_j$),
or it changes by $+1$ or $-1$ with equal probability.
Thus, after one step,
\[
\E(x_1' + x_2' + ... + x_n') = x_1 + x_2 + ... + x_n,
\]
so the expectation of the sum of all votes
(as well as the expectation of the average of all votes)
is an invariant, say $S = cn$. That is, $c$ is the initial average of
the opinions, and the expected average of the opinions at each step.

Let $1 <= i < c < i+1 <= k$.

(W.h.p, the opinions disappear in some fixed order -
assuming that we don't hit some algebraic singularities as
in configuration $(a_1, a_2, a_3)$,
where $a_1 = a_3$ and opinions 1 and 3 disappear "simultaneously".)
}

\noindent
{\bf (iii)} {In the synchronous vertex process, 
using Lemmas~\ref{lem:sum_deviation} and \ref{lem:weighted_sum_concentration},
we have $|W(T) - W(0)| = o(n)$ w.h.p., provided $T = o(1/\| \pi\|^2_2)$, which reduces to $T=o(n)$ for regular graphs.
In the asynchronous edge process,
$|S(t+1)-S(t)| \le 1$,
and in the asynchronous vertex process, 
$|Z(t+1)-Z(t)| \le n\max_{v\in V}\pi_{v}=n \|\pi\|_\infty$,
so using the Azuma-Hoeffding inequality for martingale concentration 
(Lemma~\ref{lem:Azuma-Hoeffding}),
we obtain 
$|W(T)- W(0)|\leq o(n)$
w.h.p.,
provided $T= o(n^2)$, respectively $T = o(1/\| \pi\|^2_\infty)$.
%
}
\end{proof}

In the light of Theorem \ref{ThA},  there are two main ways to analyse the problem.
For the 
expander graphs we consider, the final stage of pull voting with two values $i$ and $i+1$, 
takes $T=O(n)$ expected time in the synchronous process, or $T=O(n^2)$ expected time in the asynchronous process; 
{see e.g. \cite{AFill, CERO}}.
If the time $t$ for the other opinions to disappear, leaving only two neighbouring values, can be shown to be $t=o(T)$ then the total weight $S(t)$ is concentrated around $S(0)=cn$. 
In this case we can use Theorem \ref{ThA} to give an asymptotic result on the distribution $D$. 
There are however challenges in the analysis, e.g., 
dealing with the fact that 
a small number of opinions other than $i$ or $i+1$ 
may persist in the system for a longer time. This is the topic of \cref{Expan,KnSync}.
On graphs which are not expanders we cannot expect $S(t)$ (or $Z(t))$ to be concentrated, 
but we can try to explicitly obtain {the  distribution  of the limiting random variable,} 
at least for some special cases. This is the topic of 
Section~\ref{Path}.

{Appealing to Lemma \ref{T2} below, incremental voting on connected graphs finishes in polynomial time w.h.p..
In which case, in the limit, the winning value of an incremental voting process is given by a probability distribution $D$ on
$\{1,...,k\}$.}
As the total weight $W(t)$ is a martingale, the distribution $D$ must have expected value $c = W(0)/n$. Indeed,
for $Z(t)$ in the vertex process  (and similarly for $S(t)$ in the edge process),
\begin{equation}\label{Probs}
c=\frac{Z(0)}{n}
=\lim_{t \rightarrow \infty}\frac{\E Z(t)}{n}
= \sum_v \pi_v \lim_{t \rai} \E X_v(t)
= \sum_v \pi_v \E D
= \E D.
\end{equation}


\begin{lemma}\label{T2} {\sc Completion time, a general bound.}
For  any  
connected graph, and any of the three types of incremental voting (asynchronous or synchronous vertex process, or asynchronous edge  process), the worst-case 
expected time to eliminate one of the two
extreme opinions (over all initial configurations) is at most the worst-case
expected completion time of standard asynchronous 2-opinion
voting.
\end{lemma}

\begin{proof}
Let $A_i = \{v\in V: X_v = i\}$.
We consider our process
$(A_1(t),A_1(t), \ldots, A_k(t))_{t\ge 0}$
and the standard 2-opinion voting $B(t)_{t\ge 0}$,
where $B(t)$ and $V\sm B(t)$ are the supports of
the two opinions at time $t$.

We set $B(0) = A_{\ell}(0)$,
where $\ell$ is the minimum opinion: $\ell = \min\{\kappa: A_\kappa(0)\neq \emptyset\}$,
and couple processes $A$ and $B$, running them on
the same random selection of vertices in each step.
The two opinions in process $B$ are opinions $\ell$ and non-$\ell$, that is, for process $B$, 
each opinion other than opinion $\ell$ is viewed as the same opinion non-$\ell$.
While initially the vertices with opinion $\ell$ in process $B$ are exactly the vertices with this opinion in process $A$, this does not need to be the case later during the computation. If in the first step a vertex 
$v$ with $A$-opinion (its opinion in process $A$) equal to $q \ge \ell+2$ picks up a neighbour with 
$A$-opinion~$\ell$, then it 
updates its $A$-opinion to $q-1 \ge \ell +1$, but its $B$-opinion becomes $\ell$.

Throughout the computation, however, 
the following relations 
for the two extreme opinions $\ell$ and $r = \max\{\kappa: A_\kappa(0)\neq \emptyset\}$
hold by induction,
\begin{equation}\label{kjl3w1}
A_\ell(t)\subseteq B(t), \;\;\;
A_r(t)\subseteq V\setminus B(t)
\end{equation}
If at step~$t$ the vertex $v$ changes its $A$-opinion to $\ell$, (and consequently $v\in A_\ell(t+1)$),
this happens because $v\in A_{\ell+1}(t)$ and picks  a neighbour $w\in A_{\ell}(t)$. Then we must also have $v\in B(t+1)$,
because by induction $w\in B(t)$, so its $B$-opinion is $\ell$. 
Similarly, if $v\in A_{r-1}(t)$ changes its
$A$-opinion to $r$, this happens because $v$
picks a neighbour $w\in A_{r}(t)$; 
By induction $w\in V\setminus B(t)$, and so $v\in V\setminus B(t+1)$.

Let $T$ be the step when the two-voting process $B$ completes, that is, the first step when $B(T)$ is either empty or the whole set $V$.
In the former case,
$A_{\ell}(T) = \emptyset$ and in the latter $A_{r}(T) = \emptyset$,
from~\eqref{kjl3w1}, so by step $T$,
either opinion $\ell$ or $r$ must have 
been eliminated.
\end{proof}

\begin{corollary}\label{kl4jkd3w2}
The expected completion time of the discrete
incremental voting 
is $O(k \cdot \cT_{2-vote})$,
where $\cT_{2-vote}$ is the worst-case expected
completion time of the 2-opinion voting. 
\end{corollary}

\section{
\texorpdfstring
{Analysis of asynchronous process for $G_{n,p}$: proof of 
Theorem~\ref{ThGnp-intro}}
{Analysis of the asynchronous process for G(n,p).}
}
\label{Expan}

{%
In this section we analyse the asynchronous incremental voting on random graphs $G_{n,p}$
above the connectivity threshold, which we view as examples of expanders.
Much of the analysis is general, so the results should hold equally for other classes of graphs 
with expansion properties similar to those in Lemma \ref{Props}.  
To indicate why this should be the case, we use the example of the asynchronous edge process on a $d$-regular graph with opinions in 
$\{0,1,2\}$. The expected change in the two \lq extremal\rq~ values $0,2$ at any step is given by
\begin{equation}\label{n90s1a}
\E( N_0' + N_2')= N_0+N_2 -\frac{2}{dn} M_{0,2},
\end{equation}
where $M_{0,2}$ is the number of edges between vertices holding opinion 0 and those holding opinion 2.
This is because when an edge in $M_{0,2}$ is selected, then $N_0 + N_2$ is reduced by $1$, 
and when an edge in $M_{0,1}\cup M_{1,2}$ is selected, then the expected change of $N_0 + N_2$ is zero.


For a $d$-regular connected graph $G$ and any two vertex sets $S$ and $T$ in $G$, 
\begin{equation}\label{nhclewcb0a}
\left | e(S,T)- \frac{d|S||T|}{n} \right| \le  \l \sqrt{|S|(1-|S|/n)|T|(1-|T|/n)},
\end{equation}
where $e(S,T)$ is the number of edges between $S$ and $T$ 
and $\l$ is the absolute value of the second eigenvalue of the adjacency matrix of $G$.
%
Assume that $\l=\e d$ for some constant $\e<1$ (the expander assumption)
and take $|S|=N_0$ and $|T|=N_2$.
{If $N_0$ and $N_2$ are large, then so is $M_{0,2}=e(A_0,A_2)$, 
since from~\eqref{nhclewcb0a}, $M_{0,2}$ is close to $d N_0 N_2/n$.}
Thus from~\eqref{n90s1a}, $N_0 + N_2$ quickly decreases.
Although this type of argument may not allow us to completely eliminate one of 
the extremal values $0$ and $2$ (as $M_{0,2}$ becomes too small, eventually $0$, while both 
$A_0$ and $A_2$ are still non-empty),
we can use it to make one of $N_0$ and $N_2$ small relative to $N_1$.}

Returning to graphs $G_{n,p}$, 
we assume all opinions are in 
$\{1,2,\ldots,k\}$, for a fixed integer $k$
(constant while $n$ grows to infinity).
{The entire point of the proof in this section is to ensure that within $T = o(n^2)$ steps, 
all but $o(n)$ vertices have two adjacent opinions in $\{i,i+1\}$. 
This will allow us to apply the conclusions of Theorem \ref{ThA} (ii)-(iii), with $c' \sim c$.}

In the edge process (and in the vertex process in regular graphs) 
the expected change in the number of vertices with any given value can be characterised as follows.
Let $M_{i,j}=M_{i,j}(t)$ be the number of edges between sets $A_i$ and $A_j$ at step $t$.
Letting $N_i=N_i (t)$ and
$N_i'=N_i(t+1)$, 
\begin{flalign}
\E N_i'&= 
N_i+ \frac{1}{2m} \brac{\sum_{j \ge i+1} M_{i-1,j}+\sum_{j \le i-1} M_{i+1,j}-
\sum_{j \ne i-1,i,i+1} M_{i,j}}. \label{Ni2}
\end{flalign}
If $k>2$, then the number of vertices with an extremal value $1$ or $k$ exhibits downward drift
(the first two sums in~\eqref{Ni2} are equal to $0$ 
for $i$ equal to $1$ or $k$), provided
there are edges between $A_1$ and $\bigcup_{j\ge 3}A_j$,
or between $A_k$ and $\bigcup_{j\le k-2}A_j$.
Thus if there is enough 
connectivity (expansion) in the graph, then the support of 
one of the extremal values reduces  relatively quickly 
to $o(n)$.
If this was, say, opinion~$1$, then, still
relatively quickly, the support of the next extremal opinion, either $2$ or $k$, reduces to $o(n)$.
And so on, until the support of all opinions other than 
some two consecutive opinions $i$ and $i+1$ is reduced to $o(n)$.
The analysis of the completion of the process from such a state will requires another approach.



In what follows, for notational convenience, we use $\om, \om'$ to denote  functions tending to infinity with
$n$. It will always be the case that $\om'=o(\om)$, and in general the exact values are not important. However, in the final part of the analysis, for $n$ sufficiently large, we will
choose $\om=\log n$, and $\om'=\log \log n$.

{\bf Required properties of $G_{n,p}$.}
The following are the expansion properties of $G_{n,p}$ needed for our proofs.
The lower bound on $np$ ensures that w.h.p.\ the graph is connected.
To maintain continuity of discussion the proof of the following lemma is given in 
\cref{ProofGnpProperties}.

\renewcommand{\LemmaGNPStructure}{%
{Let $G \in G_{n,p}$, where $np \ge \log^{1+\e} n$ for some constant $\e>0$. } The following properties hold w.h.p..
\begin{enumerate}[P1.]
{ \item  {\rm (Almost regular graphs)}\\
$G$ is connected and all vertices $v$ have degree {$d(v)=np +O(\sqrt{np \log n})$} and stationary distribution $\pi_v=\frac 1n + O\brac{ \frac{1}{n\log^{\e/2} n}}$.}
\item {\rm (Large number of edges between large subsets of vertices)} \\
Let $\d\ge 5/\sqrt{np}$. For any pair of disjoint vertex sets $A$, $B$, with { $|A| \ge \d n, \; |B| \ge \d n$}, the number of edges $X_{AB}$  between $A$ and $B$ satisfies
    $\mu/2 \le  X_{AB} \le 3 \mu /2$, where $\mu=|A||B| p$,
    the expected number of edges between the sets $A$ and $B$ 
    in $G_{n,p}$.
\item {\rm (Not too many edges within small subsets of vertices)}
\begin{enumerate}[(i)]
\item For $\om \ge e$, $\om \log \om \le np$, no vertex set $S$ of size $s=n/\om$ induces more than $X_S=e^2 s^2p$ edges.
\item Let $d=np$. No set $S$, $|S| \le n/\om$ induces more than $X_S=s \sqrt{3 d \log ne/s}$ edges.
\item Provided $\om = O(\log n)$, and $np=d \ge \log^{1+\th} n$, the ratio $X_S/X_{S, V-S}$ is at most $O(1/\om)$, the value achieved in P2.(i) above.
\end{enumerate}
\end{enumerate}
}

\begin{lemma}\label{Props}
\LemmaGNPStructure
\end{lemma}

{\bf Outline to the analysis of the process.}
Our analysis of $G_{n,p}$ is for the edge process, but as w.h.p.\ vertex degrees are concentrated for the range of $p$ we consider, the vertex process and edge process are asymptotically equivalent.
{Indeed, by property P1 of Lemma \ref{Props} the degree weighted total  $Z(t)$ and the unweighted total $S(t)$
satisfy $|Z(t)-S(t)| \le c/\log^{\e/2} n$, and it suffices to analyse the convergence of $S(t)$.}

Ideally we would like to keep completely removing 
the values $\{1,...,k\}$ one by one in some order, as in the proof of Theorem \ref{ThKnSync}.
As can be seen from 
\eqref{Ni2}, the drift on the extremal values $1,k$ is negative or zero.
\begin{equation}\label{extrv}
\E N_1'+\E N_k' = N_1+N_k -\frac{1}{2m} \brac{\sum_{j \ge 3} M_{1,j}+ \sum_{j \le k-2} M_{k,j}}.
\end{equation}
Thus  at least one extremal value $1$ or $k$ should disappear, allowing us then to
 repeat the analysis with e.g., values $\{2,...,k\}$. 
 However, the time taken for such an approach is $\Om(n^2)$, which is too long  for the total weight $S(t)$ to remain concentrated around $S(0)$. 
Therefore, in our analysis, we settle for making one of the extremal values sufficiently small, which can be done in $o(n^2)$ steps and then repeat the analysis for the remaining large values. Finally one value dominates, and w.h.p. all other values disappear at some subsequent step.
It remains to be proved below that such an approach can be made to work.

The analysis proceeds in three phases, which in outline are as follows.
\begin{enumerate}[I.]
\item One by one, the extremal values are made small. 
By the beginning of iteration $r$, $1 \le r\le k-2$, the support for $r - 1 = i -1 + (k - j)$ 
extremal values $\{1,2,\ldots, i-1\} \cup \{j+1, j+3, \ldots, k\}$ has been made small, but $N_i> \d_r n$ and $N_j> \d_r n$. During iteration~$r$, the next extremal value, either $i$ or $j$, is made small.
As we progress through the iterations, our analysis loses accuracy, so 
$\d_r$ increases with~$r$ (but remains $o(1)$).
\item Exactly two adjacent values $i,i+1$ have sets of size $N_i,N_{i+1}>  n/k\om$ and $N_i+N_{i+1}=n(1-o(1))$.
\item There is a unique value $i$ with $N_i=n(1-o(1))$.
\end{enumerate}
Arriving at Phase II, the process corresponds (in general principle) to ordinary pull voting with 
two values.
If at the completion of Phase I $\min\{N_i, N_{i+1}\} < n/k\om$, we skip Phase II. 
Phase III is a clean up phase, removing any remaining small sets.
At the end of Phase III, $N_i=n$, and the analysis is completed.

\vspace{0.1in}
\noindent
{\bf Phase I. Making the first extremal value small, that is, making $N_1$ or $N_k$ small.} 

\begin{lemma}
Let $\d=\max\brac{1/n^{1/4},5/\sqrt{np}}$.
Let $T_1$ be the number of process steps to
reduce one of $A_1$ or $A_k$  to size at most $\d n$. Then
the following hold w.h.p.
\begin{enumerate}[(i)]
\item 
$\E T_1=O(n^{7/4})$.
\item
$S(T_1) \sim S(0)$.
\item $N_1(t) \le \om \d n$ at all steps $t>T_1$ (for some $\om \rai$).
\end{enumerate}
\end{lemma}
\begin{proof}
Let $A = A_1,\, B=A_k$, $|A| = N_1 > \d n$ and $|B| = N_k > \d n$.
We proceed in stages indexed by decreasing $\ell$. At the beginning of the current stage,
we assume w.o.l.g.\ that $N_1\le N_k$ and integer $\ell\ge l$ is such that 
$\ell\d n \le |A| < (\ell+1) \d n$ and $|B| \ge \ell \d n$.
We estimate the expected time of this stage, which reduces $N_1$ or $N_k$ to below $\ell \d n$. 
As $\ell$ decreases by $1$ in each stage, eventually one of $N_1$ or $N_k$ becomes less than $\d n$. 

While the sizes of both $A$ and $B$ remain at least $\ell\d n$,
by Lemma \ref{Props} property P2, $M_{1,k} \ge \ell^2\d^2n^2 p/2$ and so, from~\eqref{Ni2}, 
the expected drift of $N_1$ per
step is
\[
\E N_1'-N_1= - \frac{1}{2m} \sum_{j \ge 3} M_{1,j} \le -\ooi \frac{M_{1,k}}{n^2 p}
\le -\frac{\ell^2 \d^2}{3} = -\e.
\]


Define an {\em active step}, as a step of the process at which an edge incident with a vertex of $A_1$ is chosen.
For a biased random walk starting from $z$ on the integer line  $\{0,1,...,L\}$,
the probability $q_z$ of ruin (absorption at zero) is
\[
q_z=\frac{(q/p)^L-(q/p)^z}{(q/p)^L-1},
\]
where $q$ and $p=1-q$ are the probabilities of moving left or right, respectively, in one step.
We view $N_1$ as a random walk on the integer line $\{\ell\d n, \ldots, (\ell + 2)\d n\}$ with the starting 
position less than $(\ell + 1)\d n$. 
The probability of moving left at any (active) step is at least $q=\half(1+\e)$, and 
the probability of moving right is at most $p=\half(1-\e)$, provided that $N_k$ remains at least $\ell d n$. 
Note that $q/p=e^{c\e}$ for some absolute constant $c(\e)>2$.
%
Then $q_z$ the probability that $N_1$ reaches $\ell\d n$ before it reaches $(\ell+2)\d n$ or 
$N_k$ reaches $\ell\d n$ is at least
\begin{align*}
q_z \ge  1 - \frac{e^{c\e\d n}-1}{e^{2c\e\d n}-1} \ge 1 - e^{c\e\d n/2} = 1 - e^{c\d^3 n/6} \ge 1 - O(n^{-K}),
\end{align*}
for any $K >0$ constant.
The expected duration $\E D_\ell$ of the walk, that is, 
the expected number of active steps to reduce $N_1$ by $z \le \d n$) is
\[
\E D_{z(\ell)} \le \frac{z}{q-p} \le \frac{\d n}{\e}= \frac{3n}{\ell^2  \d}.
\]
Let $X_{A\overline{A}}$ be the number of edges from $A=A_1$ to $\overline{A} = V\setminus A$. 
By Lemma \ref{Props} property P2, at step $t$ of the current stage $\ell$,
\[
P_{\ell}=\Pr(t \text{ is an active step})
=\frac{X_{A\overline{A}}}{m}\ge \frac{|A| |B| p/2}{m} \ge \frac{(\ell\d n)(n/2)p/2}{\ooi n^2 p} 
\ge \frac{\ell \d}{5}.
\]
Let $T_1$ be the first process step at which $N_1<\d n$ (or $N_k<\d n$). Then
\[
\E(T_1)= \sum_{\ell} \frac{1}{P_{\ell}} \E D_{z(\ell)} \le
\sum_{\ell} \frac{5}{\ell  \d} \frac{3n}{\ell^2  \d}= O\brac{\frac{n}{\d^2}}
=O(n^{3/2}).
\]

As $Z$ is a martingale, 
$\E Z(t) = Z(0)$. 
Apply the Azuma martingale inequality (Lemma \ref{lem:Azuma-Hoeffding}) to  the sequence of oriented edges 
$(e_1,...,e_t)$ inspected at steps $1,\ldots,t$. 
At each step, $Z$ changes by at most $(1+o(1))$.  
At step $T^*_1 = \om \E T_1 $, with $h=\sqrt{T^*_1 \log n}$,
\begin{align*}
\Pr((|Z(T_1)-Z(0)| \ge h)\; \mbox{or} \; (T_1>T_1^*))
&\le \Pr(\exists T<T_1^*: |Z(T)-Z(0)| \ge h) +\Pr(T_1>T_1^*)\\
&\le T_1^* e^{-h^2/(3T_1^*)} + o(1) =o(1),
\end{align*}
and thus,
using \cref{Props} Property 1,
w.h.p.\ 
$S(T_1) \sim Z(T_1)\sim Z(0) \sim S(0)$ as required.

If $N_1(t)>\d n$ at some future step, then even if $N_k=0$, comparing with an unbiased random walk on $\{0,1,..., L\}$, with probability $1-1/\om$ the walk is absorbed at zero before reaching $L=\om \d n$.
\end{proof}

\vspace{0.1in}
\noindent
{\bf Phase I. Making the next 
extremal value small.}
Having completed the first iteration, we continue the induction for general iteration $r$,
$1 <r\le k-2$.
By the beginning of this iteration, the support for $r-1$
extremal values $\{1,2,\ldots, i-1\} \cup \{j+1, j+3, \ldots, k\}$ has been made small, but $N_i> \d_r n$ and $N_j> \d_r n$. During iteration $r$, the next extremal value, is made small, that is, 
$N_i$ or $N_j$ is made small.

Let $\d_1=\d$,\, $\d_1'=\om \d$, and in general $\d_r=\om^{2(r-1)}\d,\,\d_r'=\om \d_r$. 
%
The argument is similar 
as previously,
repalcing $N_1$ and $N_k$ 
with $N_i$ and $N_j$, and $\d$ with 
$\d_r$, and assuming by induction
that 
the support for the values 
not in 
$I_r = \{i,i+1,\ldots, j\}$
is at most 
$\d_{r-1}' n = \d_r n /\om$. 
\cref{Ni2} for the drift of $N_i$ can be written as
\begin{align}\label{drift}
2m \;(\E N_i'-N_i)\;=& \brac{ -\sum_{x \ne i-1,i,i+1} M_{i,x}} + \sum_{x>i} M_{i-1,x} + \sum_{x<i}
M_{i+1,x}.
\end{align}
We can assume that $N_j = \th n$
for some $j\ge i+2$, for if not, $N_i+N_{i+1}=n(1-o(1))$ 
and we can move on to Phase II.
Thus
\[
\E N_i'-N_i=  -\Om(\d_r)  + O(\d'_{r-1})=-\Om(\d_r).
\]
The induction proceeds is as before.
At the end of the last step $T_r$
of iteration $r$,
the support for opinions not in
$I_{r +1}=\{i',...,j'\}$, where 
$i' = i+1, j' = j$, or $i' = i, j' = j-1$, is at most 
$\d_r n$, and w.h.p.\ will stay 
below $\d'_r n = \om\d_r n$ for all 
subsequent steps.
Furthermore, w.h.p.\ 
$\E (T_r - T_{r-1}) = O(n/\d_r)$
and $S(T_r) \sim S(0)$.
Phase I ends at the end of 
iteration $k-2$, 
when $I_{k-1} = \{i,i+1\}$.
Recall that we assume that $k$ 
is constant.


\vspace{0.1in} \noindent
{\bf Phase II analysis.}
At some step $t$ of Phase I, w.h.p. there are values $i,\;i+1$ such that $N_i(t)+N_{i+1}(t)=n(1-o(1))$. This is the beginning of Phase II.
We remark that the analysis of the disappearing opinions doesn't show that the  two remaining opinions have values  $i=\floor{c}$ and $i+1=\ceil{c}$ w.h.p. It only shows that w.h.p. there will be at most two substantial opinions remaining at the end of Phase I. The fact that the values of $i,i+1$ must be $\floor{c}$ and $\ceil{c}$ follows from an
appeal to Theorem \ref{ThA}.

In the simplest version of  pull voting on a  regular connected graph $G$, there are initially two values $0,1$ comprising sets of sizes $pn$ and $qn$, $q = 1-p$.
The probability $0$ wins is $p$, and the probability $1$ wins is $q$, since the size of opinion $0$ is 
an unbiased random walk on integer line. 
Unfortunately this analogy is not exact.
Let $i-1$ be the last opinion reduced in Phase I.
Consideration of \eqref{drift} shows that at this point there may be a tendency for $N_{i-1}$ to increase rather than decrease. To deal with this, we will stop Phase II (and start Phase III) when $\min (N_i, N_{i+1})\le n/k\om$.

Considering value $i-1$, and replacing $i$ by $i-1$ in \eqref{drift},
the positive drift 
$\d^+\le\om\d_{k-2}$ 
may arise for example from $M_{i-2,i}$. 
If at some step $N_{i-1}/n \ge \om\d_{k-2}$, then as $N_{i+1} > n/k \om$, 
there are now a sufficient number of edges in $M_{i-1,i+1}$ for 
the drift to become again negative with value
\[
-\Om((\om \d_{k-2})(1/\om)) +O(\om\d_{k-3})=-\Om(\d_{k-2}).
\]
Thus at the end of Phase II, 
w.h.p.\ $N_{i'} \ge n-n/\om$, and
$\sum_{j \ne i'} N_j \le n/\om$,
where $i'$ is $i$ or $i+1$.

\vspace{0.1in} \noindent

{\bf Phase III analysis.}
At the start of Phase III, there is one opinion $i$ for which $N_i=n(1-o(1))$.
Let $A=A_i$, and $B=\cup_{j \ne i} A_j$.

The process resembles a 'balls in boxes' system in which a vertex with value $j$ is a box with $j$ balls, and contents of the boxes change based on edges between the boxes.
When a vertex is selected, then 
one ball may be added to, or removed from, 
this vertex, depending on the number of balls in the selected neighbour.
When all vertices have the same value (the same number of balls), the process has ended. If an edge with values $(i,j), j \ne i$ is chosen at a given step we call this a Type 1 event. Choosing an edge 
with values $(j,j'), j \ne i\ne j'$ is a Type 2 event. We ignore $(i,i)$ and $(j,j)$ events.

We compare the process with 
an unbiased walk on this integers 
$\{0,1,...,L\}$ starting from
$z=W(0)=\sum_{j\ne i } |j-i|  N_j$, the weight of set $B$ at the start of Phase III w.r.t.\ value $i$,
which represents the distance between
the starting configuration and the target 
configuration when all vertices have value~$i$.
Thus initially $z\le k|B| \le n/\om$.
Each Type 1 event changes $z$ by $+1$ or
$-1$ with equal probability
(once an edge is selected, one of its 
end vertices is chosen for an update
with equal probability).
If only Type 1 events occurred,
then $z$ would be an unbiased random walk on $\{0,1,...,L\}$, and we put
$L = n$. 
(If $i$ does not win, then
we must have $W(t)=n$ for some $t$.)
For such a walk, value $i$ wins when the walk is absorbed at zero, and the probability of this is $1-z/L$. The expected duration  of the walk is $z(L-z)$.

For steps $t=0,1,...,$ as Phase III proceeds,  the value $W(t)$ will change due to both Type 1 and Type 2 events. The change due to Type 1 events is directly included in the random walk given above, and with probability $1-\om'z/n$ the walk  will not increase above
$\om' W(0)\le \om' n/\om$, where $\om' \rai$ but $\om'=o(\om)$.

Type 2 events are not represented directly by the random walk. Each  Type 2 occurrence  can change 
the value $W$ by $+1$ or $-1$. A Type 2 event on an edge $(j,j')$ where $j \ne j'$,  increases or decreases the number of balls in the system by one with equal probability. After $T$ events of Type 2, the additional change in $W$ due to this is
 $(+1)X+(-1) (T-X)$ where $X \sim Bin(T,1/2)$. Thus $X$ will not exceed $O( \sqrt {T \log T})$, w.h.p.

Only Type 1 moves can increase the size of $B$, whereas Type 2 moves will decrease it, if 
$j=i-1$ or $j'=i+1$. 
The w.h.p.\ maximum size of $B$ due to the Type 1 walk is $s=\om' n/\om$.
By Lemma \ref{Props}, property P3, w.h.p. no set size $s$ induces more than $O(s^2 p)$ edges, whereas by property P2, there are at least $nsp/3$ edges between $A_i$ and $B$. Thus the probability of a Type 2 event is at most $O(s/n)$. Thus the number of Type 2 events in the duration of the Type 1 random walk is  of order at most
$T=zL \frac{s}{n} \om'=O((n \om'/\om)^2)$, w.h.p.. 
The w.h.p. maximum increase in $W$ due to Type 2 events is $O( \sqrt {T \log T})=O(n \om' (\log n)^{1/2}/\om)$.

Provided $\om \ge \om'^2 (\log n)^{1/2}$ we can increase $W(0)$ to $z'=n/\om'$.
Put $\om=\log n$, $\om'=\log \log n$. 
Thus with probability $1-O(1/\om')$, at the end of Phase III $N_i=n$, as required.

%


\section{Asynchronous incremental voting on the line: 
proof of Theorem~\ref{TheLine-intro}} \label{Path}

To indicate that Theorems~\ref{ThA} and \ref{ThGnp-intro} 
do not hold for general graphs, we consider the following specific example of
an {\em ordered path}.
The graph is a path with $n$ vertices $\{1,2, \dots, n\}$. There are three opinions $0, 1, 2$ and initially they are ordered along the path:
vertices $\{1, \ldots, i_0\}$ have opinion $0$,
vertices $\{i_0+1, \ldots, n-j_0\}$ have opinion $1$, and
vertices $\{n-j_0+1, \ldots, n\}$ have opinion~$2$. Thus, the $0 \le i_0 \le n$ vertices in the initial segment of the path have opinion $0$,
the $0\le j_0\le n-i_0$ vertices in the final segment of the path have opinion~$2$, and the remaining $n-(i_0+j_0)$ vertices
in the middle of the path have opinion $1$.

We show that the probability that
opinion $0$ wins is equal to $a(1-b)$, where $a = i_0/n$ and $b = j_0/n$.
By symmetry, the probability that opinion $2$ wins is equal to $(1-a)b$, leaving the probability of $ab + (1-a)(1-b)$ for opinion $1$ to win.


The non-decreasing order of the opinions along the path is an invariant of the process, so each intermediate configuration is characterised by the number  $i=N_0$  of vertices
 at the beginning of the path with opinion $0$, and the number  $j=N_2$ of  vertices at the end of the path with opinion $2$, 
 where $0 \le i \le n$, $0\le j \le n-i$. The process ends when opinion $0$ wins ($i$ becomes $n$), or opinion $2$ wins ($j$ becomes $n$), or opinion $1$ wins (both $i$ and $j$ become~$0$).  
 
The process of changing from one configuration to the next one is a random walk on the integral points of the triangle
$i \ge 0, j \ge 0, i+j \le n$;
see Figure~\ref{fig:triangle}.
Considering only the steps when the configuration changes, a configuration $(i,j)$ which is strictly inside this triangle (that is, $i > 0, j>0, i+j < n$)
changes to any of the four configurations
$(i+1,j)$, $(i-1,j)$, $(i, j+1)$ and $(i,j-1)$ with equal probability of $1/4$.
Indeed, configuration $(i,j)$ changes when
either edge $(i,i+1)$ (with opinions $0$ and $1$) or edge $(n-j,n-j+1)$
(with opinions $1$ and $2$)
is selected (equal probability).
If edge $(i,i+1)$ is selected, then the configuration changes to $(i+1,j)$ or $(i-1,j)$, depending which vertex $i$ or
$i+1$ updates its opinion. With equal probability, either vertex $i+1$ decreases its opinion from $1$ to $0$, or
vertex $i$ increases its opinion from $0$ to $1$.
Analogously when edge $(n-j,n-j+1)$ is selected.

\begin{figure}
    \begin{subfigure}[b]{0.49\textwidth}
         \centering \includegraphics[width=0.9\textwidth]{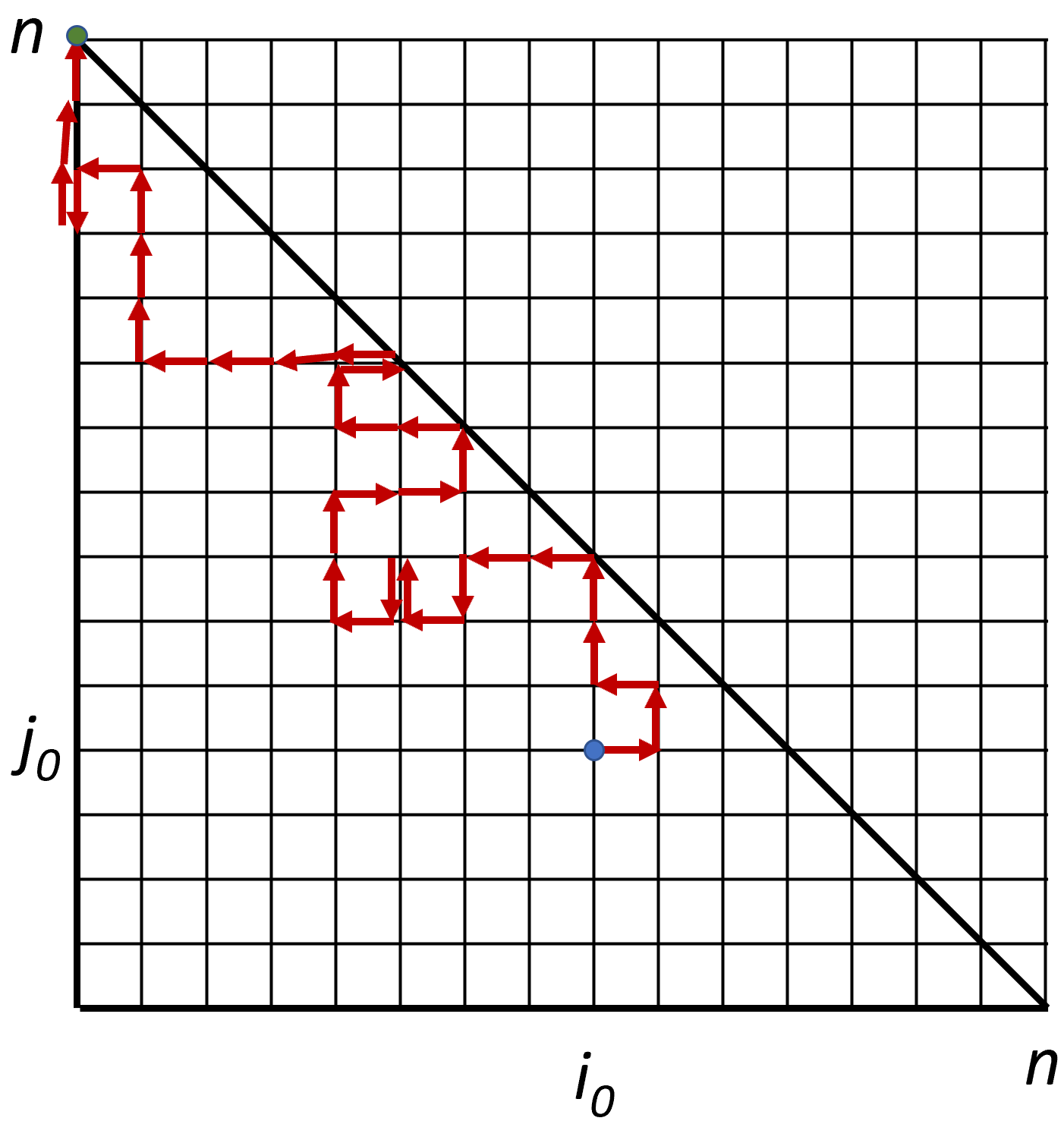}
        \caption{Random walk on the triangle grid.}
         \label{fig:triangle}
    \end{subfigure}
    \begin{subfigure}[b]{0.49\textwidth}
         \centering     \includegraphics[width=0.9\textwidth]{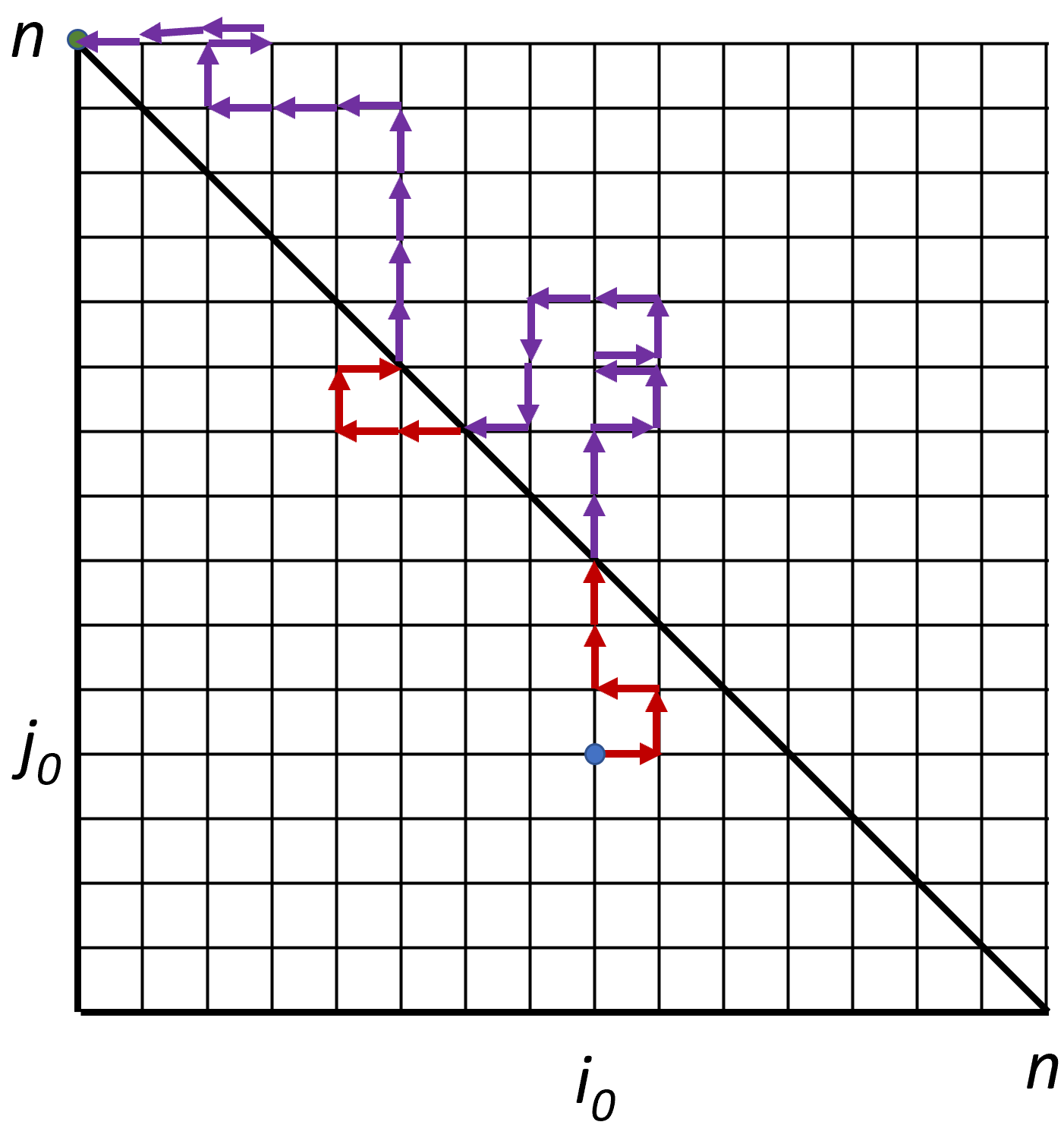}
         \caption{Random walk on the square grid.}
        \label{fig:square}
    \end{subfigure}
    \caption{The random walk on the square grid (right diagram) and the corresponding walk on the triangle grid (left diagram).
    The random walk of the triangle grid
    defines the evolution of the incremental voting on the ordered path. In this example, the averaging process starts with $i_0$ vertices with opinion $0$ and $j_0$ vertices with opinion $2$, and stabilises
    with all vertices having opinion~$2$.}
    \label{Fig:RW_for_orderedPath}
\end{figure}

From a configuration $(0,j)$, where $0<j<n$, we have
equally probable transitions to configuration $(0,j+1)$ or $(0,j-1)$.
Analogously, a configuration $(i,0)$,
where $0<i<n$, transitions to $(i+1,0)$ or
$(i-1,0)$ with equal probability.

Finally, consider the {\em diagonal} configurations lying on the side of the triangle formed by the line segment from $(0,n)$ to $(n,0)$.
For a
non-final configuration $(i,j)$: $i+j = n$, $i>0$, $j>0$,
the vertex $i$ has opinion $0$ and vertex $i+1$ has opinion $2$.
The configuration changes when the unique edge $(i,i+1)$ is selected.
In this case the configuration transitions
to $(i-1,j)$, when vertex $i$ increases its opinion from $0$ to $1$, or to $(i,j-1)$,
when vertex $i+1$ decreases its opinion from $2$ to $1$ (equal probability for either of these two transitions).

For convenience, we view this random walk $W$ on the triangle as a random walk $W'$ on the full square $0\le i \le n, 0\le j \le n$, unifying the pairs of states $(i,j)$ and $(n-j, n-i)$, these being identical on the diagonal of the triangle.
See Figure~\ref{Fig:RW_for_orderedPath},
where the right diagram gives an example of
the random walk $W'$ on the square grid, and
the left diagram shows the corresponding walk
on the triangle.
The transition probabilities for walk $W'$ are the same as for $W$ for all
non-diagonal states $(i,j)$. For such a state, if it is not on the boundary of the square, then one of the coordinates increases or decreases by $1$, with all four possibilities equally probable. For a state on the boundary of the square, the configuration changes, with equal probability, to one of the two neighbouring boundary states.

For a diagonal non-final state $(i,j)$,
the random walk $W'$ moves also to any of the four neighbouring states
$(i+1,j)$, $(i-1,j)$, $(i, j+1)$ or $(i,j-1)$, with equal probability. In this case, the transition of $W'$ with probability $1/2$ to either
$(i-1,j)$ or $(i,j+1)$ corresponds to random walk $W$ moving with probability $1/2$ from configuration $(i,j)$
to configuration $(i-1,j)$. Thus
the pair of states $(i-1,j)$ and $(i,j+1)$ in $W'$ correspond to the configuration $(i-1,j)$ in $W$.
The below diagonal state $(i,j)$ and above diagonal state $(n-j,n-i)$ in the square, both correspond to state $(i,j)$  in the triangle.

Thus our incremental voting on the path
corresponds to the random walk $W'$ on the square grid. Inside the square the walk transitions with equal probability from one state to any of the four neighbouring states. A transition on one coordinate is completely independent of the value of the other coordinate.
When the walk hits a side of the square, this corresponds to one of the two extreme values 0 or 2 being eliminated.
The walk then remains within this side of the square, moving independently to one of the two neighbouring boundary states. The final absorbing states are the four corners of the square. State $(n,0)$
corresponds to opinion $0$ winning, state
$(0,n)$ corresponds to opinion $2$ winning,
and states $(0,0)$ and $(n,n)$
(corresponding to state $(0,0)$ in the triangle) correspond to
opinion $1$ winning.

What is the probability that the random walk $W'$ terminates in the state $(n,0)$, meaning the win for opinion $0$?
We generate the two-dimensional random walk $W'$ from two independent one-dimensional walks, one walk for each of the two coordinates, both walks with the range
$\{0,1,\ldots,n\}$.
To move walk $W'$, we take, with equal probability, the next step from one of the two one-dimensional walks.
Walk $W'$ ends in the state $(n,0)$ if, and only if, the one-dimensional walk for the coordinate $i$ ends in state $n$ and
the one-dimensional walk for the coordinate $j$ ends in state $0$.
Indeed, for the 'if' part, if
the one dimensional random walks for coordinates
$i$ and $j$ end in states $n$ and $0$,
respectively, then walk $W'$ must end in the state $(n,0)$,
For the 'only if' part,
if walk $W'$ ends in $(n,0)$, then the
walk for coordinate $i$ cannot end in $0$.
Otherwise, if the
walk for coordinate $i$ ended
in $0$, then walk $W'$ would reach a state
$(0,y)$, for $0 < y < n$, and then end in either $(0,0)$ or $(0,n)$, or would reach a state $(x,0)$, for $0<x<n$ and then end
in $(0,0)$, or would reach a state
$(x,n)$, for $0<x<n$, and then end in
$(0,n)$.
Analogously,
if walk $W'$ ends in $(n,0)$, then the
walk for coordinate $j$ cannot end in $n$.

For an unbiased random walk on
$\{0,1,\ldots,n\}$ staring at position $X$,
the probability that the walk ends in the state $0$ is equal to $(n-X)/n$.
The one-dimensional random walks for the coordinates $i$ and $j$ start at positions
$i_0$ and $j_0$, respectively.
Thus the probability that the first walk
ends in $n$ is equal to $i_0/n = a$ and the probability that the second walk ends in $0$
is equal to $(n-j_0)/n = 1-b$.
%


\section{
\texorpdfstring
{Synchronous incremental voting on $K_n$: \cref{ThKnSync}}
{Synchronous incremental voting on Kn.}
}
\label{KnSync}
In this section, we show \cref{ThKnSync}, which
refers to
the synchronous process on 
the complete graph $K_n$.
To begin with, we sketch the outline of the proof.
%
{At each discrete 
time step, each vertex $v$ chooses a {vertex $w$} independently and  uniformly at random},
and updates its 
opinion $X_v$ to $X'_v$ 
as in~\eqref{Tab}.
We are interested in the evolution of $(X(t))_{t\ge 0}$. 



Firstly, we show that the smallest opinion $s=\min_{v\in V}X_v(0)$ or the largest opinion $\ell=\max_{v\in V}X_v(0)$ vanishes w.h.p.~within $O(\log n)$ steps, while $s+3\leq \ell$ (\cref{lem:extreme}).
Hence, after the smallest or largest opinion disappears $k-3$ times, which occurs w.h.p.\ in $T=O(k\log n)=o(n/\log n)$ steps, at most three consecutive opinions $\{i-1,i, i+1\}$ are left. 
Using a Martingale concentration argument
(\cref{lem:sum_deviation}) 
we further show that w.h.p.\ 
$|S(T)-S(0)| = O(\sqrt{nT\log n})=o(n)$. 

At this point only three adjacent values $\{i-1,i, i+1\}$ remain.
In  \cref{lem:threeleft}, 
we next either reduce the number of remaining opinions to two
consecutive opinions, 
or if not, and we still have three opinions, 
then the sizes of opinions 
$i-1$ and $i+1$ are $o(n)$.   
This reduction takes $o(n)$ steps w.h.p., so we still have 
$|S(T)-S(0)|=o(n)$.
In either case, the next and final phase completes in $O(n)$ expected steps, by comparison with pull voting.
The comparison is straightforward,
if only two consecutive opinions
$i$ and $i+1$ remain.
When there are three opinions 
$i-1,i,i+1$, where $|A_{i-1} \cup A_{i+1}|=o(n)$, 
then $S(t)/n \sim i$, and we prove that w.h.p.\ $i$ wins by coupling the process with pull voting.



\subsection{Many opinions case}
First, we show that one of the extreme opinions disappears within $O(\log n)$ steps.
\begin{lemma}
\label{lem:extreme}
Let $s= \min_{v\in V}X_v(0)$ and $\ell= \max_{v\in V}X_v(0)$ be the smallest and the largest opinions in the  initial round, respectively.
Suppose $\ell\geq s+3$.
Then, $N_s(T)N_\ell(T)=0$ w.h.p.~within $T=O(\log n)$ steps.
\end{lemma}
Applying \cref{lem:extreme} repeatedly, we immediately have the following.

\begin{theorem}\label{thm:SyncKn}
From any initial configuration of opinions from $[k]=\{1,2,\ldots,k\}$,
$X_v(T)\in \{i-1,i,i+1\}$ holds for some $1<i<k$ and for any $v\in V$ within $T=O(k \log n)$ steps w.h.p.
\end{theorem}
\begin{proof}[Proof of \cref{lem:extreme}]
By definition, we have that $N_s'\sim Bin\left(N_s+N_{s+1}, N_s/n\right)$
and  $N_\ell'\sim Bin\left(N_{\ell-1}+N_{\ell}, N_\ell/n\right)$.
Furthermore, $N_s'$ and $N_\ell'$ are independent since $s+1<\ell-1$.
Write $Z=N_sN_\ell$ and $Z'=N_s'N_\ell'$.
Then, we have
\begin{align}
    \E\left[Z'\right]&=\E\left[N_s'N_\ell'\right]
    =\E\left[N_s'\right]\E\left[N_\ell'\right]
    =(N_s+N_{s+1})\frac{N_s}{n}(N_{\ell-1}+N_{\ell})\frac{N_{\ell}}{n} \nonumber \\
    &= Z \frac{N_s+N_{s+1}}{n}\frac{N_{\ell-1}+N_{\ell}}{n}
    \leq Z \frac{N_s+N_{s+1}}{n}\left(1-\frac{N_s+N_{s+1}}{n}\right)
    \leq \frac{1}{4}Z.
    \label{eq:E_extreme_decrease}
\end{align}
The first inequality follows from $N_s+N_{s+1}+N_{\ell-1}+N_{\ell}\leq n$.
For $Z(t)=N_s(t)N_\ell(t)$, \eqref{eq:E_extreme_decrease} implies that $\E\left[Z(t+1)\right]\leq \E[Z(t)]/4$ holds for any $t\geq 0$.
Taking $T=\lceil 3\log n\rceil$ and using the Markov inequality, we obtain
\begin{align*}
    \Pr\left[Z(T)>0\right]
    \leq \E\left[Z(T)\right]
    \leq \frac{1}{4}\E\left[Z(T-1)\right]
    \leq \cdots \leq \frac{1}{4^T}\E[Z(0)]
    \leq \frac{n^2}{\mathrm{e}^{\lceil 3\log n\rceil }}\leq \frac{1}{n}.
\end{align*}
\end{proof}

\subsection{Difference from the initial average}
\begin{lemma}
\label{lem:sum_deviation}
Let $S(t)= \sum_{v\in V}X_v(t)$.
For any $T\geq 0$ and $\epsilon>0$,
\begin{align*}
    \Pr\left[\left|S(T)-S(0)\right|\geq \epsilon\right]\leq 2\exp\left(-\frac{\epsilon^2}{2nT}\right).
\end{align*}
\end{lemma}
\begin{proof}
{First, by \cref{SKn} we observe that $(S(t))_{t=0,1,2,...}$ is a martingale. }
From definition, we have $X_v(t+1)-X_v(t)\in \{-1,0,1\}$.
Furthermore, for any $v\neq v'$,  $X_v(t+1)-X_v(t)$ and $X_{v'}(t+1)-X_{v'}(t)$ are independent.
Write $\Delta_v(t)=X_v(t)-X_v(t-1)$.
Applying \cref{lem:Hoeffding}, we have
\begin{align}
    \E\left[\mathrm{e}^{\lambda (S(t+1)-S(t))}\middle| X(t)\right]
    &=\E\left[\mathrm{e}^{\lambda ((S(t+1)-S(t))-\E\left[S(t+1)-S(t)\mid X(t)\right])}\middle| X(t)\right]\nonumber\\
    &=\E\left[\mathrm{e}^{\lambda \sum_{v\in V}(\Delta_v(t+1)-\E\left[\Delta_v(t+1)\mid X(t)\right])}\middle| X(t)\right]
    \nonumber \\
    &=\prod_{v\in V}\E\left[\mathrm{e}^{\lambda (\Delta_v(t+1)-\E\left[\Delta_v(t+1)\mid X(t)\right])}\middle| X(t)\right]\nonumber\\
    &\leq \prod_{v\in V}\mathrm{e}^{\frac{\lambda^2}{2}}=\mathrm{e}^{\frac{\lambda^2n}{2}}. \label{eq:indep_Hoeff}
\end{align}
Combining  \eqref{eq:indep_Hoeff} and \cref{lem:AHineq}, we obtain the claim.
\end{proof}

{
{\bf Remark.} Choose $ T=\lceil 3\log n\rceil$ from the proof of \cref{thm:SyncKn}
and $\epsilon=\sqrt{7n}\log n$ in Lemma \ref{lem:sum_deviation} to obtain
\[
\Pr(|S(T)-S(0)| \ge \sqrt{7n}\log n) < \frac 1n.
\]}

Note that the $2\exp\left(-\frac{\epsilon^2}{2nT}\right)$ bound in \cref{lem:sum_deviation} is better than the $2\exp\left(-\frac{\epsilon^2}{2n^2T}\right)$ bound obtained directly from the Azuma-Hoeffding inequality (\cref{lem:Azuma-Hoeffding}).

\begin{lemma}
\label{lem:weighted_sum_concentration}
Consider a synchronous vertex process on an arbitrary graph.
Let $Z(t)= n\sum_{v\in V}\pi_v X_v(t)$.
Then, for any $T\geq 0$ and $\epsilon>0$,
\begin{align*}
    \Pr\left[\left|Z(T)-Z(0)\right|\geq \epsilon\right]\leq 2\exp\left(-\frac{\epsilon^2}{2n^2\|\pi\|_2^2T}\right),
\end{align*}
where $\|\pi\|_2=\sqrt{\sum_{v\in V}\pi_v^2}$.
\end{lemma}
\begin{proof}
Recall that $Z(t)=n\sum_{v\in V}\pi_vX_v(t)$ and $(Z(t))_{t=0,1,2,...}$ is a martingale. 
Write $\Delta_v(t)=X_v(t)-X_v(t-1)\in \{-1,0,1\}$.
Then, applying \cref{lem:Hoeffding} yields
%
\begin{align}
    \E\left[\mathrm{e}^{\lambda (Z(t+1)-Z(t))}\middle| X(t)\right]
    &=\E\left[\mathrm{e}^{\lambda ((Z(t+1)-Z(t))-\E\left[Z(t+1)-Z(t)\mid X(t)\right])}\middle| X(t)\right]\nonumber\\
    &=\E\left[\mathrm{e}^{\lambda n\sum_{v\in V}\pi_v(\Delta_v(t+1)-\E\left[\Delta_v(t+1)\mid X(t)\right])}\middle| X(t)\right]
    \nonumber \\
    &=\prod_{v\in V}\E\left[\mathrm{e}^{\lambda 
 n \pi_v(\Delta_v(t+1)-\E\left[\Delta_v(t+1)\mid X(t)\right])}\middle| X(t)\right]\nonumber\\
    &\leq \prod_{v\in V}\mathrm{e}^{\frac{4\lambda^2n^2\pi_v^2}{8}}=\mathrm{e}^{\frac{\lambda^2n^2\|\pi\|_2^2}{2}}. \label{eq:w_indep_Hoeff}
\end{align}
Combining  \eqref{eq:w_indep_Hoeff} and \cref{lem:AHineq}, we obtain the claim.
\end{proof}
For regular graphs, both $\pi_v$ and $\|\pi\|_2^2$ are $1/n$. So \cref{lem:weighted_sum_concentration} generalizes \cref{lem:sum_deviation}.

\subsection{At most three consecutive opinions remain} \label{3values}
In this section, we suppose that $X_v(0)\in \{i-1,i,i+1\}$ holds for some $i$ and for all $v \in V$, i.e., all initial opinions are from three consecutive integers.
Without loss of generality, we assume that $i=2$ throughout this section.

\begin{lemma}
\label{lem:N1N3ineq}
Suppose that $X_v(0)\in \{1,2,3\}$ holds for all $v\in V$.
Then, for any $t\geq 0$, 
    \begin{align*}
        \E[N_1(t+1)N_3(t+1)\mid X(t)]
        \leq \left(1-\frac{N_1(t)+N_3(t)}{2n}\right)N_1(t)N_3(t).
    \end{align*}
\end{lemma}
\begin{proof}
Let $Y_{i\to j}$ denote the number of vertices that change their opinion from $i$ to $j$. 
We have $N_1'=Y_{1\to 1}+Y_{2\to 1}$ and $N_3'=Y_{2\to 3}+Y_{3\to 3}$.
Note that $Y_{3\to 1}=Y_{1\to 3}=0$.
It is easy to see that $Y_{1\to 1}\sim Bin(N_1,N_1/n)$ and $Y_{3\to 3}\sim Bin(N_3,N_3/n)$.
An important observation is that $(Y_{2\to 1},Y_{2\to 2},Y_{2\to 3})$ follows a multinomial distribution with parameters $N_2$ and $(N_1/n,N_2/n,N_3/n)$.
Hence, $Cov(Y_{2\to 1},Y_{2\to 3})\leq 0$ and we have $\E[Y_{2\to 1},Y_{2\to 3}]\leq \E[Y_{2\to 1}]\E[Y_{2\to 3}]$.
Thus,
\begin{align}
\E[N_1'N_3']
&=\E[Y_{1\to 1}(Y_{2\to 3}+Y_{3\to 3})]+\E[Y_{2\to 1}Y_{2\to 3}]+\E[Y_{2\to 1}Y_{3\to 3}] \nonumber\\
&\leq \E[Y_{1\to 1}]\E[Y_{2\to 3}+Y_{3\to 3}]+\E[Y_{2\to 1}]\E[Y_{2\to 3}]+\E[Y_{2\to 1}]\E[Y_{3\to 3}] \nonumber\\
&=(\E[Y_{1\to 1}]+\E[Y_{2\to 1}])(\E[Y_{2\to 3}]+\E[Y_{3\to 3}])\nonumber\\
&=\left(N_1\frac{N_1}{n}+N_2\frac{N_1}{n}\right)\left(N_2\frac{N_3}{n}+N_3\frac{N_3}{n}\right)\nonumber\\
&=N_1\left(1-\frac{N_3}{n}\right)N_3\left(1-\frac{N_1}{n}\right). \label{eq:N1N3ineq}
\end{align}
Note that $Y_{i\to j}$ and $Y_{k\to \ell}$ are independent for $i\neq k$.
Combining \cref{eq:N1N3ineq} and the fact that $(1-x)(1-y)=1-x-y+xy\leq 1-(x+y)+\frac{(x+y)^2}{2}\leq 1-\frac{x+y}{2}$ holds for any $0\leq x+y\leq 1$, we obtain the claim. 
\end{proof}
Intuitively speaking, \cref{lem:N1N3ineq} implies that $N_1(t)N_3(t)$ continues to decrease by a factor of $1-1/\sqrt{n}$ while $N_1(t)+N_3(t)\geq 2\sqrt{n}$.
Hence, within $T=O(\sqrt{n}\log n)$ steps, $N_1(t)N_3(t)$ reaches $0$ or $N_1(t)+N_3(t)< 2\sqrt{n}$.
In other words, either of the following events occurs: 
(1) either $N_1(t)$ or $N_3(t)$ is zero, 
(2) both $N_1(t)$ and $N_3(t)$ are less than $2\sqrt{n}$.
The following lemma shows it formally.
\begin{lemma}
    \label{lem:threeleft}
    Suppose that $X_v(0)\in \{1,2,3\}$ holds for all $v\in V$.
    Let $T=\lceil 3\sqrt{n}\log n\rceil$.
    Then, for some $0\leq t\leq T$, w.h.p.\ one of the following two events occurs.
\begin{enumerate}[(1)]
    \item \label{lab:difbig} $N_1(t)=0$ or $N_3(t)=0$.
    \item \label{lab:difsmall} $N_1(t)\leq 2\sqrt{n}$ and $N_3(t)\leq 2\sqrt{n}$.
\end{enumerate}
\end{lemma}
\begin{proof}
    Let 
    \begin{align*}
        \tau&=\min\{t\geq 0\mid N_1(t)+N_3(t)<2\sqrt{n}\textrm{ or }N_1(t)N_3(t)=0\}, \\
        Y_t&=N_1(t)N_3(t)\left(1-\frac{1}{\sqrt{n}}\right)^{-t}, \hspace{1em}
        Z_t=Y_{t\wedge \tau}=Y_{\min\{t,\tau\}}.
    \end{align*}
    Note that we have $Z_{t+1}-Z_{t}=\mathbbm{1}_{\tau> t}(Y_{t+1}-Y_{t})$.
    From \cref{lem:N1N3ineq}, 
    \begin{align*}
        \E[Z_{t+1}-Z_{t}\mid X(t)]
        &=\mathbbm{1}_{\tau> t}\left(\E[Y_{t+1}\mid X(t)]-Y_{t}\right)\nonumber \\
        &\leq \mathbbm{1}_{\tau> t}\left(\frac{\left(1-\frac{N_1(t)+N_3(t)}{2n}\right)N_1(t)N_3(t)}{\left(1-\frac{1}{\sqrt{n}}\right)^{t+1}}-\frac{N_1(t)N_3(t)}{\left(1-\frac{1}{\sqrt{n}}\right)^{t}}\right)\nonumber \\
        &\leq 0 
    \end{align*}
    holds, i.e., $\E[Z_{t+1}]\leq \E[Z_t]$ and $\E[Z_T]\leq \E[Z_0]=Z_0$. 
    Let $T=\lceil 3\sqrt{n}\log n\rceil$. 
    Then,
    \begin{align}
        \E[Z_T\mid \tau>T]\Pr[\tau>T]
        \leq \E[Z_T]
        \leq Z_0
        =N_1(0)N_3(0)\leq n^2. \label{eq:superMN1N3}
    \end{align}
    Furthermore, 
    \begin{align}
        \E[Z_T\mid \tau>T]
        &=\E\left[N_1(T)N_3(T)\left(1-\frac{1}{\sqrt{n}}\right)^{-T}\mid \tau>T\right]
        \geq \left(1-\frac{1}{\sqrt{n}}\right)^{-3\sqrt{n}\log n}\geq n^3. \label{eq:lowerN1N3}
    \end{align}
    Note that the event $\tau>T$ implies that $N_1(T)+N_3(T)\geq 2\sqrt{n}$ and $N_1(T)N_3(T)\geq 1$.
    Combining \cref{eq:superMN1N3,eq:lowerN1N3}, we obtain $\Pr[\tau>T]\leq 1/n$.
\end{proof}

\subsection{
\texorpdfstring
{Completing the proof of \cref{ThKnSync}}
{Completing the proof of Theorem 1}
}
\label{subsec:completeKn}
If we have reached here at some step $t$, then  at most three values $i-1,i,i+1$ remain, and one of the cases  \cref{lem:threeleft} (1)  or \cref{lem:threeleft} (2) holds. We next prove that w.h.p.  the process will finish in $O(n)$ steps with the claimed results. 

In either case, by \cref{lem:sum_deviation}, $S(t)=S(0) (1+o(1))$. 
So if \cref{lem:threeleft} (1) holds, there are two  remaining  values, say $i,i+1$, and we can use two-value pull voting with \cref{ThA} directly.  

However if \cref{lem:threeleft} (2) holds, then there are three values $i-1,i,i+1$, where $|A_{i-1} \cup A_{i+1}|=O(n^{1/2})$. 
Thus $S(t)/n \sim i$, and we next prove that $i$ wins w.h.p. by coupling the process with pull voting.
For convenience let $\{i-1,i,i+1\}=\{1,2,3\}$, let $A_2=A$ and $B=A_1 \cup A_3$. In pull voting value $i=2$ wins with probability $|A|/n=1-o(1)$.  In one step of synchronous pull voting, $|A_P'|=Bin(n,|A|/n)$.

There is a coupling between incremental voting on three values, and pull voting such that
$|A_I'|$ stochastically dominates $|A_P'|$. Firstly the number of vertices which choose in $A=A_2$ directly is $Bin(n, |A|/n)$.  Denote this set by $A'=A'_P$ and let $|A_P'|=X_P$. Given the set $A_P'$, a further non-negative number $Y_I$ of vertices take the value $i=2$ indirectly. The value of $Y_I$ is a sum of binomials, namely
\[
Y_I=Bin(|A_3\setminus A_P'|, N_1/n) + Bin(|A_1\setminus A_P'|, N_3/n).
\]
It follows that under the coupling $|A_I'|=X_P+Y_I \ge X_P=|A_P'|$ and thus
\[
\Pr(\text{Value } i \text{ wins in incremental voting}) \ge
\Pr(\text{Value } i \text{ wins in pull voting}) =1-o(1).
\]


{ \section{Concluding comments}

The incremental voting model offers an alternative type of pull voting suitable for discrete numeric opinions which can be compared on a linear scale. This may be appropriate for systems which need
a very simple protocol which converges towards an average opinion.  As the extremal values are discarded rapidly in some instances, it could also offer a faster alternative to remove outliers in some plurality systems. 

The incremental voting process can  be viewed as a form of discrete averaging of integer weights. 
The final answer is an integer (no fractions), obtained in finite expected time.  For suitable expanders, w.h.p.\ the process returns the average rounded up or down to an integer.
To increase the accuracy of the averaging, multiply all initial values by $10^h$ before averaging. The final answer, after re-scaling, will now  be w.h.p.\ correct to the $h$-th decimal place.
The cost is the increased 
convergence time.

In incremental voting,  the weighted average remains a martingale under a wide range of conditions.  
Let $P$ be any reversible transition matrix and $\pi$ be its stationary distribution. 
Then, if the selected vertex $v$ chooses $u$ with probability
$P(v,u)$, the random variable $W=\sum_{v\in V}\pi_v X_v$ is a martingale.
As an example, if
$P(u,u)=1-L_u$ and $P(u,v)=L_u/d(u)$, 
where $0 < L_u \le 1$, then
$L_u$ can be viewed as the propensity for vertex $u$ to change its opinion 
when selected in a given step.
Here, $\pi(v)=d(v)/C L_v$, where $C=\sum L_v/d(v)$.

}

\bibliography{ref}

\newpage
\section*{Appendix}
\appendix
\section{Proof of Lemma \ref{Props}}
We repeat the lemma for convenience.
\label{ProofGnpProperties}

\begin{lemma}[Lemma \ref{Props}]\label{Props-repeat}
\LemmaGNPStructure
\end{lemma}








\begin{proof}$\;$\\
{\bf P1.}
{An application of the Chernoff-Hoeffding inequality (Lemma \ref{lem:Chernoff}) shows that for all vertices $v$,
$d(v)=np +O(\sqrt{np \log n})$.}

\noindent
{\bf P2.}
for given disjoint $A$, $B$ Let $|A|=an,\; |B|=bn$ then the Chernoff-Hoeffding inequality 
(Lemma \ref{lem:Chernoff}.\ref{symmetric})
with $\e=1/2$ implies 
\[
P_{AB}=\Pr(X_{AB} \not \in [\mu/2, 3\mu/2]) \le 2 e^{- \mu/12}.
\]
We say that $A,B$ is a {\em  bad pair}, if $|A| \ge \d n$ and $|B|\ge\d n$ but $X_{AB} \not \in [\mu/2, 3\mu/2]$. Then
\begin{align*}
\lefteqn{\E(\text{number of bad pairs}) = \sum_{A,B} P_{AB} \le} \;\;\;\;\;\;\;\; & \\ 
& \le 4^n\, 2 e^{-\d^2n^2p/12} \le 2\brac{4e^{-\d^2 np/12}}^n
\le 2\brac{4e^{-2}}^n=o(1).
\end{align*}

\noindent
{\bf P3.}
(i) Let $X_S$ denote the number of edges induced by a set $S$, and 
$\mu=\E(X_S) = \binom{s}{2}p$ the expected number. 
By the Chernoff-Hoeffding inequality (Lemma \ref{lem:Chernoff}.\ref{largeDeviation}), for $\a \ge e$ and $s\ge 3$,
\begin{equation}\label{nhcnwqekjcfh4w}
P_S= \Pr(X_S \ge \a \mu ) \le (e/\a)^{\a \mu} \le (e/\a)^{\a s^2 p /3}.
\end{equation}
Say a set $S$ of size $s$ is a {\em bad set}, if it
induces more than $e^2s^2p \geq e^2 \mu$ edges. Then, using~\eqref{nhcnwqekjcfh4w} with $\a=e^2$,
\begin{align*}
\E(\text{number of bad sets of size $s$})\le& \binom{n}{s} e^{-e^2 s^2 p /3}\\
\le &\brac{\frac{ne}{s}e^{-e^2 sp/3}}^s=\brac{\exp \{ -e^2 sp/3 + \log ne/s\} }^s\\
\le &\brac{e^{-(e^2/3) \log \om  + \log \om e }}^s 
\le \brac{e^{-(e^2/3-2) \log \om}}^s 
=o(1).
\end{align*}
The last inequality follows from the assumption that $\om \ge e$.
The size $s=n/\om$ is minimized when $\om \log \om = np$, implying that $s \ge (\log\log n)/2$. 
Sum the above over all $s$ greater than this minimum value to conclude that the expected number of bad sets
of sizes in the required range is~$o(1)$.

{
\noindent
(ii) Let $\mu=\E X_{S,V-S}=s(n-s)p$. Then, as $n-s=n (1-o(1))$, $\mu =sd(1-o(1))$ and
\[
P_{S,V-S}=\Pr (X_{S,V-S} \le (1-\e) \mu) \le e^{-\e^2 sd/3}.
\]
Thus
\begin{align*}
\E(\text{number of bad sets } S) \le \binom{n}{s}P_{S,V-S}
\le \brac{\frac{ne}{s} e^{-\e^2 d/3}}^s =o(1),
\end{align*}
provided  $\e \ge \sqrt{\frac{9 \log ne/s}{d}}$.

As the total degree of $S$ is $sd$, no such $S$ can induce as many as
\[
X_S =\e s(n-s)p \le s \sqrt{9 d \log ne/s}
\]
edges.

(iii) Thus for $s \le n/\om$,
\[
\max \frac{X_S}{X_{S,V-S}} = O \bfrac{s \sqrt{d \log n/s}}{sd}=O
 \brac{\sqrt{\frac{\log n}{d}}} =O\bfrac{1}{\log^\th n} = O \bfrac{1}{\om},
\]
provided $\om = O(\log n)$, and $np=d \ge \log^{1+\th} n$.
}



\section{
\texorpdfstring
{Asynchronous vertex process on $K_n$}
{Asynchronous vertex process on Kn}
}
\label{KnAsync}
\begin{lemma}
\label{lem:extreme_async}
Let $s= \min_{v\in V}X_v(0)$ and $\ell= \max_{v\in V}X_v(0)$ be the smallest and the largest opinions in the  initial round, respectively.
Suppose $\ell\geq s+3$.
Then, $N_s(T)N_\ell(T)=0$ w.h.p.~within $T=O(n\log n)$ steps.
\end{lemma}
\begin{proof}
Let $Z=N_sN_{\ell}$ and $Z'=N_s'N_\ell'$.
Then, we have
\begin{align}
\E\left[Z'-Z\right]
&=\E\left[N_s'N_\ell'-N_sN_{\ell}\right] \nonumber \\
&=\left((N_s+1)N_{\ell}-N_sN_{\ell}\right)\Pr\left[N_s'=N_s+1, N_\ell'=N_{\ell}\right] \nonumber \\
&+\left((N_s-1)N_{\ell}-N_sN_{\ell}\right)\Pr\left[N_s'=N_s-1, N_\ell'=N_{\ell}\right] \nonumber \\
&+\left(N_s(N_{\ell}+1)-N_sN_{\ell}\right)\Pr\left[N_s'=N_s, N_\ell'=N_{\ell}+1\right] \nonumber \\
&+\left(N_s(N_{\ell}-1)-N_sN_{\ell}\right)\Pr\left[N_s'=N_s, N_\ell'=N_{\ell}-1\right] \nonumber \\
&=N_{\ell} \frac{N_{s+1}}{n}\frac{N_{s}}{n}-N_{\ell} \frac{N_{s}}{n}\frac{n-N_{s}}{n}
+N_s \frac{N_{\ell-1}}{n}\frac{N_{\ell}}{n}-N_s \frac{N_{\ell}}{n}\frac{n-N_{\ell}}{n} \nonumber \\
&=\frac{Z}{n^2}\left(N_{s+1}-n+N_{s}+N_{\ell-1}-n+N_{\ell}\right) \label{eq:async_N1N3} \\
&\leq -\frac{Z}{n}. \label{eq:Ex_extreme}
\end{align}
Note that $N_{s+1}+N_{s}+N_{\ell-1}+N_{\ell}\leq n$ holds.
Let $Z(t)=N_s(t)N_\ell(t)$.
\eqref{eq:Ex_extreme} implies that $\E[Z(t+1)]\leq (1-1/n)\E[Z(t)]$ holds for any $t\geq 0$.
Thus, taking $T=\lceil 3n\log n\rceil$, we obtain
\begin{align*}
\Pr\left[Z(T)>0\right]
&=\Pr\left[Z(T)\geq 1\right]
\leq \E[Z(T)] \leq \left(1-\frac{1}{n}\right)\E[Z(T-1)]\\
&\leq \cdots \leq \left(1-\frac{1}{n}\right)^T\E[Z(0)]
\leq \frac{n^2}{n^3}=\frac{1}{n}.
\end{align*}
\end{proof}
The following lemma is the asynchronous version of \cref{lem:N1N3ineq}. 
\begin{lemma}
\label{lem:async_N1N3ineq}
    Suppose that $X_v(0)\in \{1,2,3\}$ holds for all $v\in V$.
Then, for any $t\geq 0$, 
    \begin{align*}
        \E[N_1(t+1)N_3(t+1)\mid X(t)]
        = \left(1-\frac{N_1(t)+N_3(t)}{n^2}\right)N_1(t)N_3(t).
    \end{align*}
\end{lemma}
\begin{proof}
Note that \cref{eq:async_N1N3} in the proof of \cref{lem:extreme_async} also holds for the case of $X_v(0)\in \{1,2,3\}$. 
Hence, 
\begin{align*}
\E\left[N_1'N_3'\right]
&=N_1N_3+\frac{N_1N_3}{n^2}\left(N_{2}-n+N_{1}+N_{2}-n+N_{3}\right)
=N_1N_3\left(1-\frac{N_1+N_3}{n^2}\right).
\end{align*} 
\end{proof}
We can easily show the following lemma in the same way of \cref{lem:threeleft}.
\begin{lemma}
    \label{lem:threeleft_asyncversion}
    Suppose that $X_v(0)\in \{1,2,3\}$ holds for all $v\in V$.
    Let $T=\lceil Cn^{1.5}\log n\rceil$.
    Then, for some $0\leq t\leq T$, either of the following events occur w.h.p.:
\begin{enumerate}
    \item \label{lab:async_difbig} $N_1(t)=0$ or $N_3(t)=0$.
    \item \label{lab:async_difsmall} $N_1(t)\leq 2\sqrt{n}$ and $N_3(t)\leq 2\sqrt{n}$.
\end{enumerate}
\end{lemma}

Using the techniques in \cref{subsec:completeKn}
for the asynchronous process on $K_n$, we obtain the asynchronous version of Theorem~\ref{ThKnSync}.
\begin{theorem}
    {\sc Asynchronous incremental voting on $K_n$}. \\
Let the initial values {of the vertices of $K_n$} be chosen from $\{1,2,\dots,k\}$, where $k=o(n/(\log n)^2)$, and 
let $S(0)= \sum_{v\in V}X_v(0)
=cn$. \;\;
\begin{enumerate}[(i)]
\item
If $i < c < i+1$, then $\,\Pr(i \text{ wins}) \sim i+1-c$ and $\,\Pr(i+1 \text{ wins}) \sim c-i$. 
If $c =i (1+o(1))$, then $\Pr(i \text{ wins}) \sim 1$.
\item
The number of opinions is reduced to at most three consecutive values 
in $O(nk\log n)$ steps w.h.p., and the expected time for the whole process to finish is $O(n^2)$.
\end{enumerate}
\end{theorem}

\section{Tools used in the analysis} 
\begin{lemma}[The Chernoff-Hoeffding inequalities]
\label{lem:Chernoff}
Let $X_1,\ldots,X_n$ be $n$ independent random variables taking values in $[0,1]$. Let $X=\sum_{i=1}^nX_i$.
Let $\mu^{-}\leq \E[X]\leq \mu^+$.
Then, we have the following:
\begin{enumerate}
\item \label{upperTail}
$\Pr\left[X\geq (1+\e)\mu^+\right]\leq \exp\left(-\frac{\min\{\e^2,\,\e\}\mu^+}{3}\right)$, for $\e\geq 0$.
\item \label{lowerTail}
$\Pr\left[X\leq (1-\e)\mu^-\right]\leq \exp\left(-\frac{\e^2\mu^-}{2}\right)$, for $0\leq \e\leq 1$.
\item \label{symmetric}
$\Pr\left[X\not\in (\, (1-\e)\mu^-, (1+\e)\mu^+\,)\right] 
    \leq 2\exp\left(-\frac{\e^2\mu^-}{3}\right)$, for $0\leq \e\leq 1$.
\item \label{largeDeviation}
$\Pr\left[X\geq \a\mu^+\right]\leq \left(\frac{e^{\a - 1}}{\a^\a} \right)^\mu$, for $\a \ge 1$.
\end{enumerate}
\end{lemma}

\begin{lemma}[The Azuma-Hoeffding inequality]
\label{lem:Azuma-Hoeffding}
Let $(X_t)_{t=0,1,2,...}$ be a martingale.
Suppose $|X_i-X_{i-1}|\leq c_i$ holds for any $i\geq 0$.
Then, for any $T\geq 0$ and $\epsilon>0$,
\begin{align*}
\Pr\left[|X_T-X_0|\geq \epsilon\right]\leq 2\exp\left(-\frac{\epsilon^2}{2\sum_{i=1}^Tc_i^2}\right).
\end{align*}
\end{lemma}
The followings are the basic technical lemmas for the Hoeffding inequality.
\begin{lemma}
\label{lem:Hoeffding}
Let $X$ be a random variable such that $\E[X]=0$ and $a\leq X\leq b$.
Then, for any $\lambda>0$, $\E\left[\mathrm{e}^{\lambda X}\right]\leq \mathrm{e}^{\lambda^2(b-a)^2/8}$.
\end{lemma}
\begin{lemma}
\label{lem:AHineq}
For any $\alpha>0$ and $t$, suppose that $\E\left[\mathrm{e}^{\alpha(Y_t-Y_{t-1})}\mid \mathcal{F}_{t-1}\right]\leq \mathrm{e}^{ \alpha^2c_t^2}$ holds for some $c_t$.
Then, for any $\epsilon>0$, $\Pr\left[|Y_T-Y_0|\geq \epsilon\right]\leq 2\exp\left(-\frac{\epsilon^2}{4\sum_{t=1}^Tc_t^2}\right)$.
\end{lemma}
\end{proof}

\end{document}